\documentclass{article}

\usepackage{amsmath}
\usepackage{amssymb}
\usepackage{amsthm}
\usepackage{mathtools}
\usepackage{geometry}
\usepackage{xcolor}
\usepackage{hyperref}

\newtheorem{thm}{Theorem}
\newtheorem{prop}[thm]{Proposition}
\newtheorem{lemma}[thm]{Lemma}
\newtheorem{corollary}[thm]{Corollary}

\theoremstyle{definition}

\newtheorem*{rem}{Remark}
\newtheorem*{rems}{Remarks}

\DeclareMathOperator{\Div}{div}

\newcommand{\NN}{\mathbb{N}}

\newcommand{\sP}{\mathcal{P}}

\newcommand{\RR}{\mathbb{R}}
\newcommand{\ZZ}{\mathbb{Z}}
\newcommand{\innpr}{\cdot}

\begin{document}

\title{Steady three-dimensional rotational flows: existence via Kato's approach to locally coercive problems
}

\author{
B. Buffoni\thanks{Institut de math\'ematiques, Station 8, Ecole Polytechnique
F\'ed\'erale de Lausanne, 1015 Lausanne, Switzerland\\boris.buffoni@epfl.ch
} 
\and E. S\'er\'e\thanks{CEREMADE, Universit\'e Paris-Dauphine, PSL Research University, CNRS, Place de Lattre de Tassigny, 75016 Paris, France\\sere@ceremade.dauphine.fr
}
}

\date{May 2025}

\maketitle

\begin{abstract}
	Stationary flows of an inviscid and incompressible fluid of constant density in the region $D=(0, L)\times \RR^2$, periodic in the second and third variables, are considered.
The flux and the Bernoulli function are prescribed at each point of the boundary $\partial D$.
	The previous existence proof relying on the Nash-Moser iteration scheme is replaced by an adaptation of Kato's approach to locally coercive problems, allowing a more precise statement: the regularity required in Sobolev spaces is the one needed to ensure a basic local coercivity property, and there is a loss of control of only two derivatives in the obtained solutions. The underlying variational structure gives an additional property: the obtained solutions are local minimizers of an integral functional.
	The strategy of proof is first developed for a simpler nonlinear partial differential equation in two variables which satisfies a weaker form of ellipticity. 
\smallskip


\noindent
{\em Mathematics subject classification (AMS, 2020):}
	35Q31, 76B47, 35G60, 47H14, 35A15.

\end{abstract}

\tableofcontents

\section{Introduction}

\subsection{On the stationary Euler equation in dimension $3$}

In \cite{Kato} (see in particular Theorem \ref{thm: Kato} below), Kato showed how degree theory can replace the Nash-Moser iteration scheme in some examples. Following his ideas, we revisit the local existence problem for stationary inviscid flows, as studied in \cite{BuWa} with the help of the Nash-Moser approach.
When the velocity field $v$ is independent of time,
the Euler equation for an inviscid and incompressible fluid
of constant density is given by
$$(v\innpr \nabla)v=-\nabla p,~~\Div v=0,$$
where in this paper the domain is $D=(0,L)\times \RR^2$ with $L>0$,
and the stationary flows are sought periodic in the second and third variables, so that the cell of the periodic lattice is
$$\sP=(0,L)\times (0,P_y)\times(0,P_z)$$
with given  periods $P_y, P_z>0$.
Any constant vector field $\overline v$ is a solution on 
$D$ with constant pressure $\overline p$.
We do not use the Nash-Moser method as in \cite{BuWa}, but the functional spaces are similar, with improvement on the required regularity. 
Let a pair of functions $(f,g)$ be called {\em admissible} (see \cite{Bu:2012, BuWa} and also \cite{Ke}) if
\begin{itemize}
\item
$f$ and $g$ are  of class $C^2(\overline D)$,
\item
$\nabla f$ and $\nabla g$ are $P_y$-periodic in $y$ and $P_z$-periodic in $z$,
\item
$\displaystyle~
(f(x,y,z),g(x,y,z))=(\widetilde f_0(x,y,z),\widetilde g_0(x,y,z)),~\text{for all } (x,y,z)
\in \{0,L\}\times \RR^2,$
\end{itemize}
where $\widetilde f_0$ and $\widetilde g_0$ are two fixed functions
of class $C^2(\overline D)$
such that $\nabla \widetilde f_0$ and $\nabla \widetilde g_0$ 
are $P_y$-periodic in $y$ and $P_z$-periodic in $z$.
The corresponding velocity field
$v=\nabla f\times \nabla g$ is divergence free and its first component
$$v_1=(\nabla f\times\nabla g)\innpr (1,0,0)
=\partial_y f\, \partial_z g-\partial_yg\, \partial_z f
=\partial_y \widetilde f_0\,\partial_z \widetilde g_0-\partial_y\widetilde g_0\,
\partial_z \widetilde f_0$$
is prescribed on $\{0,L\}\times \RR^2$. 
We also assume that
the function $H\colon \RR^2\rightarrow \RR$ is of class $C^1$ and
that $\partial_f H$ and $\partial_g H$  composed with every admissible pair  $(f,g)$ are $(P_y,P_z)$-periodic in $y$ and $z$ (see below the statement of Theorem \ref{thm: hydrodynamics} and the remarks that follow it for a more explicit formulation). 

We are looking for an admissible pair $(f,g)$ satisfying the equations
\begin{equation}\label{eq: deux}
\left(\begin{array}{c}
-\Div(\nabla g \times (\nabla f\times \nabla  g))+
\partial_f H( f, g)
\\
-\Div((\nabla  f\times \nabla  g)\times \nabla f))
+\partial_gH(f,g)
\end{array}\right)=0,
\end{equation}
so that $v \coloneqq \nabla  f\times \nabla g$
is a solution to the Euler equation with 
$p=-\frac 1 2|v|^2
+H(f,g)$.
This has the following variational interpretation. Let the admissible pair $(f_c,g_c)$ be a critical point (in the sense that all directional derivatives at $(f_c,g_c)$  vanish) of the integral functional
\begin{equation}
\label{eq: integral functional}
\int_{\sP}\Big\{ \frac12|\nabla f\times \nabla g|^2+H(f,g)\Big\}\,dx\,dy\,dz
\end{equation}
defined on the set of admissible pairs $(f,g)$.
Then $(f_c,g_c)$ satisfies \eqref{eq: deux}
and $\nabla  f_c\times \nabla g_c$
is a solution to the Euler equation.
See e.g. \cite{BuWa} for more details.

In the following theorem, 
$$H^n_{loc}(D)=\{h:D\rightarrow\RR^\ell\,|\, h\in H^n((0,L)\times]-k,k[^2)\text{ for all }k\in \NN^*\},$$ 
where $\ell\geq 1$ is often implicit. In other words, the index "loc" is for the $y$ and $z$ variables only, not for the $x$ variable.
The aim of the paper is to provide a simpler proof and improve the minimal regularity required in \cite{BuWa} as follows:
the regularity conditions analogous to  $H^7_{loc}$ and $H^5_{loc}$ below are,
in Theorem 1.1 in \cite{BuWa}, of the type $H^{13}_{loc}$ and $H^6_{loc}$.

\begin{thm}
\label{thm: hydrodynamics}
Let us fix $L,P_y,P_z>0$ and affine functions $\overline f$ and $\overline g$ on $D$ such that
$$\overline v_{1}\neq 0 
~~\text{ with }~~ \overline v=\nabla \overline f\times \nabla \overline g=(\overline v_{1},\overline v_{2},\overline v_{3}).$$
Consider
a  pair $(\widetilde f_0,\widetilde g_0)\in H^7_{loc}(D)$ such that $\widetilde f_0-\overline f$ and $\widetilde g_0-\overline g$ are $(P_y,P_z)$-periodic in $y$ and $z$, 
and a function $H:\RR^2\rightarrow \RR$ of class $C^6$ such that $\nabla H$ admits the vectorial periods $(P_y\partial_2\overline f,P_y\partial_2 \overline g)$ and $(P_z\partial_3\overline f,P_z\partial_3\overline g)$ (these two vectors are constant and there is no periodicity condition on $H$ itself):
\begin{equation}
\label{eq: periodicity}
\forall (u,v)\in \mathbb R^2~~\nabla H(u+P_y\partial_2 \overline f,v+P_y\partial_2 \overline g)=\nabla H(u,v)=\nabla H(u+P_z\partial_3 \overline f,v+P_z\partial_3 \overline g).
\end{equation}

Assume that, for some small enough  $\xi\in(0,1]$,
\begin{equation}
\label{eq: estimates in r square}
||(\nabla \widetilde f_0-\nabla \overline f,\nabla \widetilde g_0-\nabla \overline g)||_{H^6(\sP)}\leq \xi r
~~\text{ and }~~ 
||\nabla H||_{W^{5,\infty}(\RR^2)}\leq \xi r
\end{equation}
for some $r\in]0,1]$ as small as needed (in a sense that can depend on $\xi$).
Define $H_{0,per}^5(\sP)$ as the Banach space made of pairs  $(F,G)\in H^5_{loc}(D)$ such that $(F,G)$ vanishes on $\{0,L\}\times \RR^2$ and is $(P_y,P_z)$-periodic in $y$ and $z$.
We endow $H_{0,per}^5(\sP)$ with the standard norm, the square of which is defined by
\begin{equation*}
||(F,G)||_{H^5(\sP)}^2:=\sum_{\alpha_1+\alpha_2+\alpha_3\leq 5}||\partial_{x}^{\alpha_1}\partial_y^{\alpha_2}\partial_z^{\alpha_3} F||_{L^2(\sP)}^2
\end{equation*}
for non-negative integers $\alpha_1,\alpha_2,\alpha_3$.
In \eqref{eq: estimates in r square}, if $\xi\in]0,1]$ is chosen small enough, and $(\widetilde f_0,\widetilde g_0)$ and $H$ 
are chosen such that \eqref{eq: estimates in r square} is satisfied for some small enough $r\in]0,1]$ (in a sense that depends on $\xi$), then there exists a solution $(F,G)\in \overline B(0,r)\subset H_{0,per}^5(\sP)$ of 
$$
\left(\begin{array}{c}
-\Div(\nabla (\widetilde g_0+G) \times (\nabla (\widetilde f_0+F)\times \nabla (\widetilde g_0+G)))+
\partial_f H(\widetilde f_0+F,\widetilde g_0+G)
\\
-\Div((\nabla (\widetilde f_0+F)\times \nabla (\widetilde g_0+G))\times \nabla (\widetilde f_0+F)))
+\partial_gH(\widetilde f_0+F,\widetilde g_0+G)
\end{array}\right)=0
$$
(that is, \eqref{eq: deux} with $(f,g)=(\widetilde f_0+F,\widetilde g_0+G)$).
\end{thm}

\begin{rems}
For all $(F,G)\in H_{0,per}^5(\sP)$, it easily checked that \eqref{eq: periodicity} implies that $\nabla H(\widetilde f_0+F,\widetilde g_0+G)$ is $(P_y,P_z)$-periodic in $y$ and $z$.
Moreover note that if $(F,G)\in H_{0,per}^5(\sP)$, then $(\widetilde f_0+F,\widetilde g_0+G)$ is an admissible pair, in the sense defined above, and of class $C^3$, the regularity required when bounding from below the quadratic functional of Theorem \ref{thm: quadratic part}. Finally, let us mention that the solution found in Theorem \ref{thm: hydrodynamics} is locally unique: this will be proved in Section \ref{section: minimizer} where a variational interpretation will be given.
\end{rems}

\vspace{2mm}

\subsection{A simpler problem in dimension $2$}
To explain the method, we now introduce a simpler two-dimensional PDE.
Let 
$$a_1=a_1(x_1,x_2,z_1,z_2)= a_1(x,z) \text{ and }  a_2=a_2(x_1,x_2,z_1,z_2)=a_2(x,z)\in \RR,
$$
$$x_1\in[0,1],~x_2,z_1,z_2\in \RR,$$ 
be two given functions of four real variables, supposed of class $C^{s+1}$ with $s\geq 5$, having the period $1$ with respect to $x_2$.
We shall consider
$$(x_1,x_2)\rightarrow \Big(a_1\Big(x_1,x_2,\partial_{x_1}u(x_1,x_2),\partial_{x_2}u(x_1,x_2)\Big),a_2\Big(x_1,x_2,\partial_{x_1}u(x_1,x_2),\partial_{x_2}u(x_1,x_2)\Big)\Big)$$
$$=a(x,\nabla u(x))\in \RR^2,$$
where the (unknown) function $u$ depends on the variables $x_1,x_2$, is periodic in $x_2$, the periodicity domain being noted $\sP=(0,1)^2$ ($x_1\in(0,1)$, period $1$ in $x_2$).
We assume $u\in H^{s}_{0,per}(\sP)$ (homogeneous Dirichlet's condition at $x_1=0$ and $x_1=1$, and periodicity in $x_2$, no Dirichlet's condition for the first order partial derivatives).
All first order partial derivatives of $u$ are continuous on $\overline \sP:=[0,1]\times (0,1)$ ($x_1\in[(0,1]$, period $1$ in $x_2$), thanks to the Sobolev injection $H^{s-1}_{per}(\sP)\subset C(\overline\sP)$.

We shall deal with the following problem: given $h\in H^{s}_{per}(\sP)$, find $u:[0,1]\times \RR\rightarrow\RR$ such that
\begin{equation}
\label{eq: the toy equation}
\operatorname{div}(a(x,\nabla u))=h,~~u(x_1,x_2+1)=u(x_1,x_2),~~u(0,x_2)=u(1,x_2)=0,
\end{equation}
where $u\in Y=Y_s$, with
$$s\geq 5 \text{ (an integer)},~~Y=Y_s=H^s_{0,per}(\sP),~~Y^*=Y_s^*=Y' \text{ (topological dual)},~~
$$
$$
~~V=V_s=H^{s+2}_{0,per}(\sP),~~V^*=V_s^*=V'\,,$$
and the two canonical dualities $<\cdot,\cdot>_{Y\times Y^*}$ and $<\cdot,\cdot>_{V\times V^*}$ are compatible with the $L^2(\sP)$ scalar product.
In contrast with usual elliptic problems, $u$ is sought of the same regularity as $h$, and not in $V$ ("loss of control of two derivatives").

In addition, we suppose the following "symmetry" property 
\begin{equation}
\label{eq: symmetry}
\partial_{z_p}a_i=\partial_{z_i}a_p ~\text{ for all }~ i,p\in\{1,2\}
\end{equation}
and that there exists $\rho>0$ such that, for all $z\in \RR^2$ with $|z|\leq \rho$, all $x_1\in[0,1]$ and all $x_2\in\RR$, 
\begin{equation}
\label{eq: hyp on a}
a(x_1,x_2,0,0)=(0,0)~~\text{ and }~~\forall p\in \RR^2~~p^TJ_a(x,z)p\geq \rho p_1^2,
\end{equation}
where $J_a(x,z)$ is the Jacobian matrix of $a$ with respect to $(z_1,z_2)$ at $(x,z)$ (symmetric by \eqref{eq: symmetry}). In \eqref{eq: hyp on a}, only $p_1^2$ occurs in the right-hand side of the inequality and thus the Jacobian matrix $J_a(x,z)$ is not necessarily positive definite.
As observed by Kato \cite{Kato} in his context, the hypothesis $a(x_1,x_2,0,0)=(0,0)$ can be released by replacing $h$ with $h-\operatorname{div}a(x_1,x_2,0,0)$ and $a$ with $a-a(x_1,x_2,0,0)$.
Note that $\eqref{eq: symmetry}$, which is used in Paragraph \ref{subsection: KoNi} to check \eqref{eq: conclusion} below, implies that the problem has a variational structure. This will allow us in Section \ref{section: minimizer} to interpret the obtained solutions as local minimizers.

With the help of Kato's approach presented in \cite{Kato}, we shall prove the following theorem for the toy problem, which is in the spirit of Theorem \ref{thm: hydrodynamics} above.
\begin{thm}
\label{thm: main}
Under conditions \eqref{eq: symmetry} and \eqref{eq: hyp on a}, let $r>0$ be small enough and $s\geq 5$.
If 
\begin{equation} \label{eq: smallness}
\left\|\max_{|z|\leq \rho,\,|\mu|\geq 1,|\mu|+|\nu| \leq s+1}|\partial^{(\mu_1,\mu_2,\nu_1,\nu_2)} a(x,z)|\,\right\|_{L^{2}(\sP)}~~\text{ is small enough}
\end{equation}
($\rho>0$ as in \eqref{eq: hyp on a}), then
for any  $h\in H^{s}_{per}(\sP)$ with $||h||_{H^{s}(\sP)}$ small enough,
there exists $u\in Y$ such that $\operatorname{div}(a(x,\nabla u))=h$ and $||u||_Y\leq r$.
If $a$ does not depend on $x$ (that is, $a=a(z)$), then \eqref{eq: smallness} holds trivially.
\end{thm}
The point is that the partial differential equation is not elliptic, and thus it is not possible, in a perturbative argument for example, to rely on standard results for elliptic problems. However its linearization $\operatorname{div}(J_a(x_1,x_2,0,0)\nabla u)=0$ is of the form studied by Kohn and Nirenberg in \cite{KoNi}, and the present problem might be studied with the Nash-Moser method (see \cite{Ki,BuWa} for more complicated nonlinear problems of the same kind handled with the Nash-Moser method). Our aim is to avoid the latter method, but instead to rely on the more "elementary" approach advocated by Kato. As this toy problem has a simple structure, other methods are probably possible, but this partial differential equation is only considered here in view of the more involved application to hydrodynamics of Section \ref{section: hydro}. In particular, it might be possible to get some kind of (weak) solution by minimization, under appropriate assumptions. The next theorem is the main result of \cite{Kato} stated in the special case of the present context:
\begin{thm}[Kato \cite{Kato}]
\label{thm: Kato}
Let $K$ be  the closed ball in $Y$ of radius $r>0$ centered at the origin $0\in Y$, and consider a weakly sequentially continuous map $A:K\subset Y \rightarrow  V^*$  such that 
\begin{equation}
\label{eq: Kato coercivity}
<v,Av>_{V\times V^*}\, \geq \, \beta\geq 0~~\text{ for all }~~v\in V\cap \partial K.
\end{equation}
Then $A(K)$ contains the closed ball centered at the origin $0\in Y^*$ of radius $\beta r^{-1}$ in $Y^*$ (seen as a subset of $V^*$: $V\subset Y$ and $Y^*\subset V^*$ in the standard way, see \eqref{eq: map J}).
\end{thm}
"Weakly continuous" means that if $y_j$ converges weakly to $y_0$ in $Y$, then $Ay_j$ converges weakly to $Ay_0$ in $V^*$, see the last paragraph on p. 926 in \cite{Kato}.
The fact that the assumption is for all $v\in V\cap \partial K$ instead of all $v\in \partial K$ again reflects that there is a loss of control of two derivatives in the conclusion.
In addition Kato gives two illuminating examples to illustrate how to apply his theorem; however our choice \eqref{eq: choice of A} of $A$ is more intricate than the ones in his examples. An important ingredient to check assumption
\eqref{eq: Kato coercivity} when proving Theorem \ref{thm: main} is the following local coercivity inequality for some constant $C_{lc}>0$:
$$
\int_{\sP}\sum_{i=1}^2\Big(\partial_{x_2}^s\Big(a_i(x_1,x_2,\partial_{x_1} F,\partial_{x_2}F)\Big)\Big)\partial_{x_2}^s\partial_{x_i}F\, dx_1 dx_2
$$
$$
\geq  C_{lc}  ||\partial_{x_2}^sF||_{L^2(\sP)}^2+  C_{lc}  ||\partial_{x_1}\partial_{x_2}^sF||_{L^2(\sP)}^2- C_{lc} ^{-1}||F||_{H^{s}(\sP)}^3
- C_{lc} ^{-1}\left\|\max_{|z|\leq \rho}|\partial_{x_2}^{s+1}a(x,z)|\,\right\|_{L^2(\sP)}||\partial_{x_2}^{s-1}\nabla F||_{L^2(\sP)}
$$
\begin{equation}
\label{eq: conclusion}
- C_{lc} ^{-1}\left(\left\|\max_{|z|\leq \rho}|\partial_{x_2}^2\nabla_z a(x,z)|\,\right\|_{L^\infty(\sP)} ||\partial_{x_2}^{s-1}\nabla F||_{L^2(\sP)} +||F||_{H^{s-1}(\sP)}\right) ||\partial_{x_2}^{s-1}\nabla F||_{L^2(\sP)}
\end{equation}
for all $F\in H^{s+2}_{0,per}(\sP)$ with $||F||_{H^s(\sP)}$ small enough.
Its proof is different from the analogous ones in the examples in \cite{Kato}; as it is involved,  Subsection \ref{section: proof of crucial inequality} is devoted to it.
The first example in \cite{Kato} consists in a first order equation introduced by Moser in \cite{Moser}, and the second example consists in a higher-order elliptic problem with more derivatives in the nonlinear term than in the linear one, first considered by Rabinowitz in \cite{ Rabinowitz}. Both examples are for periodic conditions and were dealt with initially by the Nash-Moser iteration method \cite{Moser,Rabinowitz}.
In the present work, the Dirichlet condition for $x_1=0$ and $x_1=1$ is the source of new difficulties. To cope with the fact that (higher-order) partial derivatives do not satisfy this Dirichlet condition in general, we make use of the topological dual $(H_{0,per}^s(\sP))'$ of $H_{0,per}^s(\sP)$ and of the dual of the partial derivative $\partial_{x_1}$ when classical integration by parts in $x_1$ is not allowed.
Another feature is that the dependence of certain high-order partial derivatives of $a$ are required to be small enough when $a$ depends on $x_1$ and $x_2$; see \eqref{eq: smallness}.
Finally \eqref{eq: estimates in r square} in the hydrodynamic problem is analogous to condition \eqref{eq: smallness} in the toy problem.

\section{The toy problem}\label{Toy}
\subsection{Proof of Theorem \ref{thm: main}}

Theorem \ref{thm: main} will be proved as a consequence of Theorem \ref{thm: Kato}. Lemma \ref{lemma: weak sequential continuity}  and Proposition \ref{prop: verification semi-coercivity} will allow us to check the hypotheses in Theorem \ref{thm: Kato}.
Then, thanks to Proposition \ref{prop: Aw=0 w=0} we will see that the conclusion of Theorem \ref{thm: Kato} implies the one of Theorem \ref{thm: main}.

The standard injections (usually seen as inclusions)
\begin{equation}
\label{eq: map J}
J:V\rightarrow Y,~~J':Y^*\rightarrow V^*
\end{equation}
are continuous and dense.
Moreover, for all $v\in V$ and $y_*\in Y^*$,
$$<Jv,y_*>_{Y\times Y^*}=<v,J'y_*>_{V\times V^*}$$
by the definition of $J'$; thus the two dualities are compatible as required in \cite{Kato}: if $J$ and $J'$ are not explicitly written, we get indeed
$<v,y_*>_{Y\times Y^*}=<v,y_*>_{V\times V^*}$.

By the compact Sobolev injection $H_{0,per}^s(\sP)\subset C^{s-2}_{per}(\overline{\sP})$, the linear operators
$$\partial_{x_1}: Y\rightarrow C_{per}^{s-3}(\overline \sP)~~\text{ and }~~\partial_{x_2}: Y\rightarrow C_{per}^{s-3}(\overline \sP)$$
are compact.
Thus, for $w\in Y$, 
$$\nabla w\in C_{per}^{s-3}(\overline \sP)^2\,,~~ a(x,\nabla w)\in C_{per}^{s-3}(\overline \sP)^2\,,~$$
$$
\operatorname{div}(a(x,\nabla w))=(\partial_{x_1}a_1)(x,\nabla w)+(\partial_{x_2}a_2)(x,\nabla w)+\operatorname{tr} \Big(J_a(x,\nabla w)\operatorname{Hess}_{w}\Big)\in C_{per}^{s-4}(\overline\sP)\subset L^2(\sP),$$
$$\partial_{x_2}^{s-1}(a(x,\nabla w))\in (L^2(\sP))^2,  ~~~~~\text{(cf Lemma \ref{lemme: controle a} in the Appendix of Section \ref{Toy})}$$
where appear the Jacobian matrix $J_a$ of $a$ with respect to the two last variables $(z_1,z_2)$, and the Hessian matrix of $w\,.$
Moreover, for all $w\in Y$,
\begin{equation}
\label{eq: choice of A}
Aw:=\underbrace{(\partial_{x_1}^s)^*\underbrace{\partial_{x_1}^{s-2}\operatorname{div}(a(x,\nabla w))}_{\in L^2(\sP)}}_{\text{first term}}
-\underbrace{\lambda_1(-1)^s \partial_{x_2}^{s}\operatorname{div}\partial_{x_2}\underbrace{\partial_{x_2}^{s-1}(a(x,\nabla w))}_{\in L^2(\sP)^2}}_{\text{second term}}
-\underbrace{\lambda_2\operatorname{div}(\underbrace{a(x,\nabla w)}_{\in L^2(\sP)^2})}_{\text{third term}}\in V^* 
\end{equation}
for two constants $\lambda_1,\lambda_2>0$ to be carefully chosen later.
In the definition of $A$,
\begin{itemize}
\item first term:
$(\partial_{x_1}^s)^* :L^2(\sP)=L^2(\sP)'\rightarrow V'=V^*$ is the dual operator of  $\partial_{x_1}^s: V\rightarrow  L^2(\sP)$, that is,
$$\forall f\in L^2(\sP)~\forall v\in V~~<v,(\partial_{x_1}^s)^*f>_{V\times V'}\,:=
\int_{\sP}(\partial_{x_1}^sv)f\, dx_1dx_2;$$
Moreover 
$\partial_{x_1}^{s-2}\operatorname{div}(a(x,\nabla w))\in L^2(\sP)$ (cf Lemma \ref{lemme: controle div a});
\item second term:
$\partial_{x_2}^{s}\operatorname{div}\partial_{x_2}$ is the dual of\footnote{This term is differentiated successively $s+2$ times in  $x_2$.}
$(-1)^s\partial_{x_2}\nabla\partial_{x_2}^s:V\rightarrow L^2(\sP)^2$ and $\partial_{x_2}^{s-1}(a(x,\nabla w))\in L^2(\sP)^2=(L^2(\sP)^2)'$ (cf Lemma \ref{lemme: controle a}); here we have used the standard notation for derivatives of distributions);
\item third term:
$\operatorname{div}$ is the dual of $-\nabla: V\rightarrow L^2(\sP)^2$ and $a(x,\nabla w)$ is seen in $L^2(\sP)^2=(L^2(\sP)^2)'$.
\end{itemize}

\begin{lemma}
\label{lemma: weak sequential continuity}
If $w_n\rightharpoonup w$ in $Y$, then $Aw_n \rightharpoonup Aw$ in $V^*$.
\end{lemma}
\begin{proof}
If $w_n\rightharpoonup w$ in $Y$, then $\nabla w_n\rightarrow \nabla w$ in $C^{s-3}(\overline\sP)^{2}$. Moreover, for all $v\in V$, we have
$$\langle v,Aw_n\rangle_{V\times V^*}
=\int_{\sP}\left(\partial_{x_1}^s v\right)\left(\partial_{x_1}^{s-2}\operatorname{div}(a(x,\nabla w_n))\right)\, dx_1dx_2
$$$$
-\lambda_1\int_{\sP}\left(\partial_{x_2}\nabla\partial_{x_2}^{s} v\right)\cdot \left(\partial_{x_2}^{s-1}(a(x,\nabla w_n))\right)\, dx_1dx_2
+\lambda_2\int_{\sP}(\nabla v)\cdot a(x,\nabla w_n)\, dx_1dx_2,
$$
$$\partial_{x_1}^{s-2}\operatorname{div}(a(x,\nabla w_n))\rightharpoonup
\partial_{x_1}^{s-2}\operatorname{div}(a(x,\nabla w)),~~
\partial_{x_2}^{s-1}(a(x,\nabla w_n))\rightharpoonup \partial_{x_2}^{s-1}(a(x,\nabla w)),~~
$$$$
a(x,\nabla w_n)\rightharpoonup a(x,\nabla w)
$$
weakly in  $L^2(\sP)$ or $L^2(\sP)^2$.
For the first and second weak limits, see Lemmas  \ref{lemme: controle div a} and \ref{lemme: controle a} in the Appendix of Section \ref{Toy}.
\end{proof}

Before stating Proposition \ref{prop: verification semi-coercivity}, let us give some preliminary local coercivity estimates derived from \eqref{eq: hyp on a}. These estimates will be used in the proofs of the more involved estimate \eqref{eq: conclusion} and of Proposition \ref{prop: verification semi-coercivity}.
\begin{prop}
\label{prop: semi-coerc} Under hypothesis
\eqref{eq: hyp on a} (and assuming that $s\geq 5$), the three following inequalities 
\begin{equation}
\label{eq: equation coercivite}
\int_{\sP}\sum_{i=1}^2\sum_{p=1}^{2}(\partial_{z_p}a_i(x,\nabla F))\Big(\partial_{x_2}^s\partial_{x_p}F\Big)\Big(\partial_{x_2}^s\partial_{x_i}F\Big)\, dx_1 dx_2 \geq  C_{lc}  ||\partial_{x_2}^sF||_{L^2(\sP)}^2+  C_{lc}  ||\partial_{x_1}\partial_{x_2}^sF||_{L^2(\sP)}^2\,,
\end{equation}
\begin{equation}
\label{eq: simplified semi-coer}
\int_{\sP}(\nabla F)\cdot a(x,\nabla F)\, dx_1dx_2
=\int_{\sP}\sum_{i=1}^2a_i(x,\nabla F)\partial_{x_i}F\, dx_1 dx_2 \geq \rho ||\partial_{x_1}F||_{L^2(\sP)}^2
\end{equation}
and
$$\int_{\sP}\left(\partial_{x_1}^s F\right)\left(\partial_{x_1}^{s-2}\operatorname{div}(a(x,\nabla F))\right)\, dx_1dx_2
\geq  C_{lc} ||\partial_{x_1}^{s}F||_{L^2(\sP)}^2
$$$$
- C_{lc} ^{-1}\left( ||\partial_{x_1}^{s-1}\partial_{x_2} F||_{L^2(\sP)}+ ||\partial_{x_1}^{s-2}\partial_{x_2}^2 F||_{L^2(\sP)}+||F||_{H^{s}(\sP)}^2\right)||\partial_{x_1}^{s}F||_{L^2(\sP)}
$$
\begin{equation}
\label{eq: consequence of semi-coer}
- C_{lc} ^{-1}\left(||F''||_{H^{s-3}(\sP)}+||(\partial^{s-2}_{x_1}\nabla_x a)(x,\nabla F)||_{L^2(\sP)}\right)||\partial_{x_1}^{s}F||_{L^2(\sP)}
\end{equation}
hold for all $F$ in $H^{s+2}_{0,per}(\sP)$ with $r=||F||_{H^s(\sP)}$ small enough, and some constant $ C_{lc} >0$.
\end{prop}

\noindent
{\em Remark:} \eqref{eq: equation coercivite} is needed in the proof of  \eqref{eq: conclusion} in Subsection \ref{section: proof of crucial inequality} (the constant $ C_{lc} >0$ needs not be the same
in \eqref {eq: conclusion} and \eqref{eq: equation coercivite}).

\begin{proof}
\noindent{\em Proof of \eqref{eq: equation coercivite}.}
The compact Sobolev injection $H_{0,per}^s(\sP)\subset C^{s-2}_{per}(\overline{\sP})$ ensures that $\nabla F$ is uniformly near $(0,0)$ if $r$ is small enough. Therefore  one gets the pointwise inequality
$$\sum_{i=1}^2\sum_{p=1}^{2}(\partial_{z_p}a_i(x,\nabla F))\Big(\partial_{x_2}^s\partial_{x_p}F\Big)\Big(\partial_{x_2}^s\partial_{x_i}F\Big)
\stackrel{\eqref{eq: hyp on a}}
\geq  \rho  |\partial_{x_1}\partial_{x_2}^sF|^2$$
and thus 
$$
\int_{\sP}\sum_{i=1}^2\sum_{p=1}^{2}(\partial_{z_p}a_i(x,\nabla F))\Big(\partial_{x_2}^s\partial_{x_p}F\Big)\Big(\partial_{x_2}^s\partial_{x_i}F\Big)\, dx_1 dx_2 \geq  \rho ||\partial_{x_1}\partial_{x_2}^sF||_{L^2(\sP)}^2\,.
$$
Observe that $\partial_{x_2}^sF$ also satisfies the Dirichlet condition at $x_1=0$ and $x_1=1$ and thus the Poincar\'e inequality gives \eqref{eq: equation coercivite}.

\noindent{\em Proof of \eqref{eq: simplified semi-coer}.}
For each fixed $x\in \overline\sP$,
$$a_i(x,\nabla F)-a_i(x,0,0)=\int_0^1\sum_{p=1}^2 \partial_{z_p}a_i(x,t\nabla F)\partial_{x_p}F\,dt,$$
and therefore, as $a_i(x,0,0)=0$, and using \eqref{eq: hyp on a},
$$\int_{\sP}\sum_{i=1}^2a_i(x,\nabla F)\partial_{x_i}F\, dx_1 dx_2 
\geq \int_{\sP}\int_0^1\sum_{p=1}^2\sum_{i=1}^2  \partial_{z_p}a_i(x,t\nabla F)\Big(\partial_{x_p}F\Big)\Big(\partial_{x_i}F\Big)\, dt \,dx_1 dx_2 \geq \rho \Vert\partial_{x_1}F\Vert^2_{L^2(\sP)}.$$

\noindent{\em Proof of \eqref{eq: consequence of semi-coer}.}
Observe that, in
$$\operatorname{div}(a(x,\nabla F))=(a_1)_{x_1}(x,\nabla F)+(a_1)_{z_1}(x,\nabla F)F_{x_1,x_1}+(a_1)_{z_2}(x,\nabla F)F_{x_1,x_2}$$
\begin{equation}\label{eq: dvpt}
+(a_2)_{x_2}(x,\nabla F)+(a_2)_{z_1}(x,\nabla F)F_{x_1,x_2}+(a_2)_{z_2}(x,\nabla F)F_{x_2,x_2}\,,
\end{equation}
the factor $(a_1)_{z_1}$ in front of $F_{x_1,x_1}=\partial^2_{x_1}F$ is such that $(a_1)_{z_1}(x,\nabla F)>0$ if $|\nabla F|$ is small enough, by \eqref{eq: hyp on a}.

In \eqref{eq: consequence of semi-coer}, the term $- C_{lc} ^{-1}||F||_{H^{s}(\sP)}^2||\partial_{x_1}^{s}F||_{L^2(\sP)}$ comes from terms in
the development of $\partial_{x_1}^{s-2}\operatorname{div}(a(x,\nabla F))$ of the form
$$g_{\mu,\nu}(x,\nabla F)(\partial^{\alpha_1}\widetilde F_1)\cdots (\partial^{\alpha_k}\widetilde F_k)\,,~~\widetilde F_j\in\{\partial_{x_1}F,\partial_{x_2 }F\},$$
where 
$$\text{$g_{\mu,\nu}$ is a multiple of $\partial^{(\mu,\nu)}a_1$ or $\partial^{(\mu,\nu)}a_2$,}$$
$\alpha_1,\ldots,\alpha_k\in \mathbb N^2$ are non-trivial multi-indices, $\mu,\nu\in \mathbb N^2$,
$$|\mu|+|\alpha_1|+\ldots+|\alpha_k|=s-1,~~ |\nu|=k$$
and, here, $k\in\{2,\ldots,s-1\}$.
These terms are estimated with the help of \eqref{eq: Kato bis} in the Appendix of Section \ref{Toy}, observing that $\widetilde F_j\in  H^{s-1}_{per}(\sP)\subset L^{\infty}(\sP)$ because $s\geq 3$:
$$\left\|\partial^{(\mu,\nu)}a_i(x,\nabla F)(\partial^{\alpha_1}\widetilde F_1)\cdots (\partial^{\alpha_k}\widetilde F_k)\right\|_{L^2(\sP)}
$$$$
\leq\operatorname{Const}\left(\max_{|\tilde \mu|+|\tilde \nu|\leq s-1,|\tilde \mu|\leq s-1-2}||(\partial^{(\tilde \mu,\tilde \nu)}a_i)(x,\nabla F)||_{L^\infty(\sP)}\right)||F'||_{L^\infty(\sP)}^{k-1}|| F'||_{H^{s-1}(\sP)}
.$$
See the proof of Lemma \ref{lemme: controle a} for additional explanation about the differentiation of a composition. 

The analogous terms corresponding to $k=|\nu|=1$ and $\mu=0$ are the terms
$$(a_1)_{z_1}(x,\nabla F)\partial_{x_1}^{s}F+\Big((a_1)_{z_2}(x,\nabla F)+(a_2)_{z_1}(x,\nabla F)\Big)\partial_{x_1}^{s-1}F_{x_2}+(a_2)_{z_2}(x,\nabla F)\partial_{x_1}^{s-2}F_{x_2,x_2}\,,$$
which give the first three terms in the right-hand side of \eqref{eq: consequence of semi-coer}.

The term $- C_{lc} ^{-1}||F''||_{H^{s-3}(\sP)}||\partial_{x_1}^{s}F||_{L^2(\sP)}$ comes from analogous terms where 
$k=|\nu|=1$ and $|\mu|\geq 1$,  estimated as follows (since $1\leq \alpha_1\leq s-2$ and $|\mu|\leq s-2$):
$$\left\|\partial^{(\mu,\nu)}a_i(x,\nabla F) \partial^{\alpha_1}\widetilde F_1 \right\|_{L^2(\sP)}
\leq\operatorname{Const}\left(\max_{1\leq |\tilde \mu|\leq s-2,|\tilde \nu|=1}||(\partial^{(\tilde \mu,\tilde \nu)}a_i)(x,\nabla F)||_{L^{\infty}(\sP)}\right)||F''||_{H^{s-3}(\sP)}\,.$$

Finally the term  $- C_{lc} ^{-1}||(\partial^{s-2}_{x_1}\nabla_x a)(x,\nabla F)||_{L^2(\sP)}||\partial_{x_1}^{s}F||_{L^2(\sP)}$ corresponds to the case $k=|\nu|=0$ and $|\mu|=s-1$.
\end{proof}

\begin{prop}
\label{prop: verification semi-coercivity}
There exist $\lambda_1,\lambda_2,L>0$ such that for all $w$ in $V$ with $||w||_Y$ small enough (norm in $Y$ instead of  $V$):
$$\langle w,Aw\rangle_{V\times V^*}\geq L ||w||_Y^2
\, -\, L^{-1}\left\|\max_{|z|\leq \rho,\,|\mu|\geq 1,|\mu|+|\nu| \leq s+1}|\partial^{(\mu_1,\mu_2,\nu_1,\nu_2)} a(x,z)|\,\right\|_{L^{2}(\sP)}||w||_Y\,.$$
\end{prop}
\begin{proof} The proof uses in a crucial way the estimate \eqref{eq: conclusion} which will be proved in Subsection \ref{section: proof of crucial inequality}.
Let $\lambda_1,\lambda_2>0$ be chosen later.
For $s\geq 5$, $w\in V$ and $r:=||w||_{Y}$ small enough:
$$\langle w,Aw\rangle_{V\times V^*}
=\Big\langle w,(\partial_{x_1}^s)^*\partial_{x_1}^{s-2}\operatorname{div}(a(x,\nabla w))
-\lambda_1(-1)^s \partial_{x_2}^{s}\operatorname{div}\partial_{x_2}\partial_{x_2}^{s-1}(a(x,\nabla w))
$$$$
-\lambda_2\operatorname{div}(a(x,\nabla w))\Big\rangle _{V\times V^*}
\stackrel {Corr.~\ref{cor: controle a}} =\int_{\sP}\left(\partial_{x_1}^s w\right)\left(\partial_{x_1}^{s-2}\operatorname{div}(a(x,\nabla w))\right)\, dx_1dx_2
$$$$
+\lambda_1\int_{\sP}\left(\nabla\partial_{x_2}^{s} w\right)\cdot \left(\partial_{x_2}^{s}(a(x,\nabla w))\right)\, dx_1dx_2
+\lambda_2\int_{\sP}(\nabla w)\cdot a(x,\nabla w)\, dx_1dx_2
$$$$
\stackrel{\eqref{eq: consequence of semi-coer},\eqref{eq: conclusion},\eqref{eq: simplified semi-coer}}
\geq  C_{lc} ||\partial_{x_1}^{s}w||_{L^2(\sP)}^2 - C_{lc} ^{-1}\left( ||\partial_{x_1}^{s-1}\partial_{x_2} w||_{L^2(\sP)}+ ||\partial_{x_1}^{s-2}\partial_{x_2}^2 w||_{L^2(\sP)}+||w||_{H^{s}(\sP)}^2\right)||\partial_{x_1}^{s}w||_{L^2(\sP)}
$$$$
- C_{lc} ^{-1}\left(||w''||_{H^{s-3}(\sP)}+||(\partial^{s-2}_{x_1}\nabla_x a)(x,\nabla w)||_{L^2(\sP)}\right)||\partial_{x_1}^{s}w||_{L^2(\sP)}
$$
$$
+\lambda_1\left(
 C_{lc} ||\partial_{x_2}^sw||_{L^2(\sP)}^2+ C_{lc} ||\partial_{x_1}\partial_{x_2}^{s}w||_{L^2(\sP)}^2- C_{lc} ^{-1}r^3- C_{lc} ^{-1}\left\|\max_{|z|\leq \rho}|\partial_{x_2}^{s+1}a(x,z)|\,\right\|_{L^2(\sP)}||\partial_{x_2}^{s-1}\nabla w||_{L^2(\sP)}\right)
$$
$$
-\lambda_1 C_{lc} ^{-1}\left(\left\|\max_{|z|\leq \rho}|\partial_{x_2}^2\nabla_z a(x,z)|\,\right\|_{L^\infty(\sP)} ||\partial_{x_2}^{s-1}\nabla w||_{L^2(\sP)} +||w||_{H^{s-1}(\sP)}\right) ||\partial_{x_2}^{s-1}\nabla w||_{L^2(\sP)}
$$
$$
+\lambda_2\rho||\partial_{x_1}w||_{L^2(\sP)}^2.
$$
By \eqref{eq: a Sobolev interpolation} and  Lemma \ref{eq: mixed derivatives} of the Appendix of Section \ref{Toy}, the norms defined by the square roots of
$\sum_{i+j\leq s}||\partial_{x_1}^i\partial_{x_2}^jw||_{L^2(\sP)}^2$ 
and 
$||w||_{L^2(\sP)}^2+||\partial_{x_1}^sw||_{L^2(\sP)}^2+||\partial_{x_2}^sw||_{L^2(\sP)}^2$ 
are equivalent on $Y$.
Thus
$$\langle w,Aw\rangle_{V\times V^*}
\stackrel{\eqref{eq: a Sobolev interpolation}}
\geq \frac 1 2   C_{lc} ||\partial_{x_1}^{s}w||_{L^2(\sP)}^2 
-D ||\partial_{x_2}^{s}w||_{L^2(\sP)}^2	-D||w||_{L^2(\sP)}^2
+\lambda_1
 C_{lc} ||\partial_{x_2}^sw||_{L^2(\sP)}^2
$$$$
-(1+\lambda_1)D||w||_Y\left\|\max_{|z|\leq \rho,\,|\mu|\geq 1,|\mu|+|\nu| \leq s+1}|\partial^{(\mu_1,\mu_2,\nu_1,\nu_2)} a(x,z)|\,\right\|_{L^{2}(\sP)}
$$
$$
-\lambda_1 C_{lc}^{-1}||w||_{H^{s-1}(\sP)}\Big(||w||_{L^2(\sP)}+||\partial_{x_1}^sw||_{L^2(\sP)}+||\partial_{x_2}^sw||_{L^2(\sP)}\Big)
$$$$
+\lambda_2 \rho ||\partial_{x_1}w||_{L^2(\sP)}^2  - (1+\lambda_1)D r^3
$$
for some constant $D>0$, and hence
$$\langle w,Aw\rangle_{V\times V^*}\geq L||w||_Y^2
\, -\, L^{-1}\left\|\max_{|z|\leq \rho,\,|\mu|\geq 1,|\mu|+|\nu| \leq s+1}|\partial^{(\mu_1,\mu_2,\nu_1,\nu_2)} a(x,z)|\,\right\|_{L^{2}(\sP)}||w||_Y$$
for  $\lambda_1,\lambda_2>0$ chosen large enough (and by considering small enough $r>0$), thanks to Lemma \ref{eq: mixed derivatives}  and the Poincar\'e inequality.
In these inequalities, the positive term $ C_{lc} ||\partial_{x_1}\partial_{x_2}^{s}w||_{L^2(\sP)}^2$ in \eqref{eq: conclusion} has not been used.
\end{proof}

\begin{prop}
\label{prop: Aw=0 w=0}
Let $\lambda_1,\lambda_2>0$, $w\in Y$, $h\in H^s_{per}(\sP)$ and define $f\in Y^*\subset V^*$ by
$$<v,f>_{Y\times Y^*}=\int_{\sP}(\partial_{x_1}^sv)(\partial_{x_1}^{s-2}h)\, dx_1dx_2
-\lambda_1\int_{\sP}(\partial_{x_2}^sv)(\partial_{x_2}^{s}h)\, dx_1dx_2
-\lambda_2\int_{\sP}vh\, dx_1dx_2\,,~~\forall v\in Y\supset V.$$
If $Aw=f$ in $V^*$,  then $\operatorname{div}a(x,\nabla w)=h$.
\end{prop}
\begin{proof} 
To prove that $\widehat h:=h-\operatorname{div}(a(x,\nabla w))=0$, observe that (see \eqref{eq: choice of A})
$$0=<v,f-Aw>_{V\times V^*}$$
\begin{equation}
\label{eq: choice of A reformulated}
=\!\!\int_{\sP}(\partial_{x_1}^sv)(\partial_{x_1}^{s-2}\widehat h) dx_1dx_2
-\lambda_1\int_{\sP}(\partial_{x_2}^{s+2}v)(\partial_{x_2}^{s-2}\widehat h) dx_1dx_2
-\lambda_2\int_{\sP}v\widehat h\, dx_1dx_2,\ \forall v\in V.
\end{equation}
In the second term of the right-hand side, $\operatorname{div}$  has been moved among a succession of operators; this is allowed if $v\in V\cap C^\infty([0,1]\times\mathbb R)$ and hence for  $v\in V$ by density of smooth functions.
In other words, the aim is to show that the linear map 
$$-(\partial_{x_1}^s)^*\partial_{x_1}^{s-2}+\lambda_1 (-1)^s\partial_{x_2}^{2s}+\lambda_2\,:\, H^{s-2}_{per}(\sP)\rightarrow V^*=(H^{s+2}_{0,per}(\sP))^*$$ 
is injective.

In this order, for $s\geq 5$ and  $\sigma>0$, let $L_\sigma$ be the differential operator
$$u\rightarrow L_\sigma u= (-1)^{s}u^{(2s)}-\sigma u^{(2)},$$
$$D(L_\sigma)=\{u\in H^{2s}(0,1):\, u=u^{(s)}=u^{(s+1)}=\ldots=u^{(2s-2)}=0\text{ at }x=0\text{ and }x=1 \}\subset L^2(0,1).$$
It is self-adjoint (see e.g. \cite{GGK}, XIV.4, Theorem 4.4) and, as $L_\sigma$ is injective, its range is dense.
To check that $L_\sigma$ is indeed injective, let $u$ be in its kernel; then
$$0=\int_0^1 (L_\sigma u)(x)u(x)\, dx
=\int_0^1 \Big(u^{(s)}(x)u^{(s)}(x)+\sigma u'(x)u'(x)\Big)\, dx
$$
and hence $u=0$. As a consequence the range of  $L_\sigma^*=L_\sigma$ is dense, and the theory of differential operators ensures that $L_\sigma$ is even surjective and that  it is a bicontinuous bijection from $D(L_\sigma)$ to $L^2(0,1)$
(see \cite{GGK}, XIV.4, Proposition 3.4).
The norm of its inverse $L_\sigma^{-1}:L^2(0,1)\rightarrow L^2(0,1)$ is denoted by $||L_\sigma^{-1}||$, and the one of its inverse $L_\sigma^{-1}:L^2(0,1)\rightarrow H^{2s}(0,1)$ by $||L_\sigma^{-1}||_{L^2\rightarrow H^{2s}}$.

Let $\widetilde h\in H^{s-2}_{per}(\sP)\subset H^{2}_{per}(\sP)$ be such that
$$\partial_{x_1}^2\widetilde h=\widehat h=h-\operatorname{div}(a(x,\nabla w))\in H^{s-2}_{per}(\sP),~~\widetilde h(0,x_2)=\widetilde h(1,x_2)=0.$$ 
For all	$v\in V$, we have (see \eqref{eq: choice of A reformulated})
\begin{equation}
\label{eq: choice of A reformulated again}
0=\int_{\sP}(\partial_{x_1}^sv)(\partial_{x_1}^{s}\widetilde h)\, dx_1dx_2
-\lambda_1\int_{\sP}(\partial_{x_2}^{s+2}v)(\partial_{x_2}^{s-2}\partial_{x_1}^2\widetilde h)\, dx_1dx_2
-\lambda_2\int_{\sP}v\partial_{x_1}^2\widetilde h\, dx_1dx_2\,.
\end{equation}
To check that $\widetilde h=0$,
write $\widetilde h$ and $v$ as a Fourier series in $x_2$:
$$\widetilde h(x_1,x_2)=\sum_{k\in \mathbb Z}\gamma_k(x_1)e^{2k\pi i x_2},~~\gamma_k(x_1)=\int_0^1e^{-2k\pi i x_2}h(x_1,x_2)\, dx_2=\overline{\gamma_{-k}(x_1)} ,$$
and 
$$v(x_1,x_2)=\sum_{k\in \mathbb Z}\chi_k(x_1)e^{2k\pi i x_2},~~\chi_k(x_1)=\int_0^1e^{-2k\pi i x_2}v(x_1,x_2)\, dx_2=\overline{\chi_{-k}(x_1)},$$
where
$$||\widetilde h||_{L^2(\sP)}^2 =\sum_{k\in \mathbb Z}||\gamma_k||_{L^2(\sP)}^2<\infty
~~\text{ and }~~
||v||_{L^2(\sP)}^2 =\sum_{k\in \mathbb Z}||\chi_k||_{L^2(\sP)}^2<\infty.$$
Define
$$
\chi_k(x_1)=\frac{1}{(1+(2k\pi )^2)^{s}(1+||L_{\sigma_k}^{-1}||_{L^2\rightarrow H^{2s}})}(L_{\sigma_k}^{-1}\gamma_k)(x_1),~~k\in \mathbb Z,
$$
where $\sigma_k=\lambda_1(2k\pi)^{2s}+\lambda_2$,
that is
$$v(x_1,x_2)
=\sum_{k\in \mathbb Z}\frac{1}{(1+(2k\pi )^2)^{s}(1+||L_{\sigma_k}^{-1}||_{L^2\rightarrow H^{2s}})}(L_{\sigma_k}^{-1}\gamma_k)(x_1)e^{2k\pi i x_2}.$$

Then, for all $(\alpha_1,\alpha_2)\in \NN^2$ such that $\alpha_1,\alpha_2\leq 2s$,  
$$||\partial^{(\alpha_1,\alpha_2)}v||_{L^2(\sP)}^2
=\sum_{k\in \mathbb Z}\frac{|2\pi k|^{2\alpha_2}}{(1+(2k\pi )^2)^{2s}(1+||L_{\sigma_k}^{-1}||_{L^2\rightarrow H^{2s}})^2}\left\|\frac{d^{\alpha_1}}{dx_1^{\alpha_1}}L_{\sigma_k}^{-1}\gamma_k\right\|_{L^2(\sP)}^2<\infty.$$
Thus $v\in V$,
$$
0\stackrel{\eqref{eq: choice of A reformulated again}}=\int_{\sP}(-1)^s(\partial_{x_1}^{2s}v)\widetilde h\, dx_1dx_2
-\lambda_1\int_{\sP}(-1)^{s}(\partial_{x_2}^{2s}v)\partial_{x_1}^2\widetilde h\, dx_1dx_2
-\lambda_2\int_{\sP}v\partial_{x_1}^2\widetilde h\, dx_1dx_2
$$$$
=\sum_{k\in \ZZ}\int_0^1\overline{\gamma_k(x_1)}L_{\sigma_k}\left(\frac{1}{(1+(2k\pi )^2)^{s}(1+||L_{\sigma_k}^{-1}||_{L^2\rightarrow H^{2s}})}(L_{\sigma_k}^{-1}\gamma_k)(x_1)\right)\,dx_1
$$$$
=\sum_{k\in \ZZ}\int_0^1\frac{1}{(1+(2k\pi )^2)^{s}(1+||L_{\sigma_k}^{-1}||_{L^2\rightarrow H^{2s}})}|\gamma_k(x_1)|^2\,dx_1
$$
and therefore $\widetilde h=0$.
\end{proof}

We are now ready to complete the proof of Theorem \ref{thm: main}.

\begin{proof}[Proof of Theorem \ref{thm: main}]
By Proposition \ref{prop: verification semi-coercivity},  there exist $\lambda_1,\lambda_2, r>0$ such that
$$\beta:=\inf\left\{<w,Aw>_{V\times V^*}:\, w\in V\,, \, ||w||_Y=r \right\}\geq \frac L 2 r^2>0$$
if
$$L^{-1}\left\|\max_{|z|\leq \rho,\,|\mu|\geq 1,|\mu|+|\nu| \leq s+1}|\partial^{(\mu_1,\mu_2,\nu_1,\nu_2)} a(x,z)|\,\right\|\leq \frac 1 2 L r .$$
By Lemma \ref{lemma: weak sequential continuity},
	$A:Y \rightarrow  V^*$  is weakly sequentially continuous.
The hypotheses of Theorem \ref{thm: Kato} being satisfied with  the closed ball
$$K=\{w\in Y:\, ||w||_Y\leq r\},$$ 
$A(K)$ contains the closed ball centered at the origin $0\in Y^*$ of radius $\beta r^{-1}>0$ in $Y^*$.
In particular, let $h\in H^s_{per}(\sP)$  be such that
$$f:=
(\partial_{x_1}^s)^* \partial_{x_1}^{s-2}h
-\lambda_1\Big((-1)^{s}\partial_{x_2}^{s}\Big)\partial_{x_2}^{s}h
-\lambda_2h\in Y^*$$
satisfies $||f||_{Y^*} \leq \beta r^{-1}$ (this holds if $||h||_{H^s_{per}(\sP)}$ is small enough).
Here $(-1)^s\partial_{x_1}^s: L^2(\sP)\rightarrow Y^*$ and $(-1)^{s}\partial_{x_2}^{s}: L^2(\sP)\rightarrow Y^*$  are the dual maps of, respectively, 
$\partial_{x_1}^s: Y\rightarrow L^2(\sP)$ and
$\partial_{x_2}^{s}: Y\rightarrow L^2(\sP)$.
Moreover $\lambda_2h\in L^2(\sP)$ is seen in $Y^*$ via the dual map of the inclusion $Y\subset L^2(\sP)$.

Hence there exists $u\in K$ such that $Au=f$.
By Proposition \ref{prop: Aw=0 w=0}, $\operatorname{div}( a(x,\nabla u))=h$.
\end{proof}

\subsection{Proof of  the local coercivity inequality \eqref{eq: conclusion}}
\label{section: proof of crucial inequality}
The following estimates are similar to the ones written in \cite{BuWa} for the hydrodynamic problem below. However, they are developed here for the toy problem and closer to the spirit of Kato's paper \cite{Kato}, which should make them clearer. 

In the sequel, we assume a priori estimates in $H^{s}_{0,per}(\sP)$ rather than in $H^{s+2}_{0,per}(\sP)$,  even if $F$ is in the latter space.
We get that  $a_i(x,\nabla F)$ is periodic in $x_2$ and
\begin{equation}
\label{eq: integrale de base}
\int_{\sP}\sum_{i=1}^2\Big(\partial_{x_2}^s\Big(a_i(x,\nabla F)\Big)\Big)\partial_{x_2}^s\partial_{x_i}F\, dx_1 dx_2\,=\,I_1+I_2+I_3+I_4+I_5+I_6+I_7+I_8
\end{equation}
with
\begin{align*}
I_1&:=\int_{\sP}\sum_{i=1}^2\sum_{p=1}^{2}(\partial_{z_p}a_i)\Big(\partial_{x_2}^s\partial_{x_p}F\Big)\Big(\partial_{x_2}^s\partial_{x_i}F\Big)\, dx_1 dx_2
\\
	I_2&:=\int_{\sP}s \sum_{i=1}^{2}\sum_{q=1}^2\sum_{p\neq i}(\partial_{z_p}\partial_{z_q}a_i)\Big(\partial_{x_2}^{s-1}\partial_{x_p}F\Big)\Big(\partial_{x_2}\partial_{x_q}F\Big)\Big(\partial_{x_2}^s\partial_{x_i}F\Big)\, dx_1 dx_2
\\
I_3&:=\int_{\sP}s\sum_{i=1}^2\sum_{q=1}^{2}(\partial_{z_i}\partial_{z_q}a_i)\Big(\partial_{x_2}^{s-1}\partial_{x_i}F\Big)\Big(\partial_{x_2}\partial_{x_q}F\Big)\Big(\partial_{x_2}^s\partial_{x_i}F\Big)\, dx_1 dx_2
\\
I_4&:=\int_{\sP}\sum_{\substack{k\geq 2,i,\alpha,\beta \\|\alpha|=|\alpha_1|+\ldots+|\alpha_k|=s \\ 1\leq \alpha_1,\ldots,\alpha_k\leq s-2}}a_{i,\alpha,\beta}(x,\nabla F)\Big(\partial_{x_2}^{\alpha_1}\partial_{x_{\beta_1}}F\Big)\ldots\Big(\partial_{x_2}^{\alpha_k}\partial_{x_{\beta_k}}F\Big)\Big(\partial_{x_2}^s\partial_{x_i}F\Big)\, dx_1 dx_2
\\
I_5&:=\sum_{i=1}^2\int_{\sP}(\partial_{x_2}^sa_i)\Big(\partial_{x_2}^s\partial_{x_i}F\Big)\, dx_1 dx_2
\\
I_6&:=\int_{\sP}s \sum_{i=1}^{2}\sum_{p\neq i}(\partial_{x_2}\partial_{z_p}a_i)\Big(\partial_{x_2}^{s-1}\partial_{x_p}F\Big)\Big(\partial_{x_2}^s\partial_{x_i}F\Big)\, dx_1 dx_2
\\
I_7&:=\int_{\sP}s\sum_{i=1}^2(\partial_{x_2}\partial_{z_i}a_i)\Big(\partial_{x_2}^{s-1}\partial_{x_i}F\Big)\Big(\partial_{x_2}^s\partial_{x_i}F\Big)\, dx_1 dx_2
\\
I_8&:=\int_{\sP}\sum_{\substack{\mu\geq 1,|\nu|=k\geq 1,i,\alpha,\beta \\\mu+|\alpha|=s \\ 1\leq \alpha_1,\ldots,\alpha_k\leq s-2}}a_{i,\mu,\nu,\alpha,\beta}(x,\nabla F)\Big(\partial_{x_2}^{\alpha_1}\partial_{x_{\beta_1}}F\Big)\ldots\Big(\partial_{x_2}^{\alpha_k}\partial_{x_{\beta_k}}F\Big)\Big(\partial_{x_2}^s\partial_{x_i}F\Big)\, dx_1 dx_2
\end{align*}
In $I_1$--$I_4$, $a_i$ is not differentiated with respect to $x_2$, whereas in $I_5$--$I_8$ it is differentiated at least once with respect to $x_2$.
In $I_4$, it is supposed that the $k$ components of  $\beta$ belong to $\{1,2\}$ and are in a non-decreasing order.

In $I_2$ and $I_3$, the factor
$s$ is obtained for $s\geq 5$ by the multivariate Faa Di Bruno formula, see e.g. \cite{CoSa}.

In $I_4$, $a_{i,\alpha,\beta}$ is a multiple of a partial derivative of $a_i$ of order $k\leq s$ with respect to $z_1,z_2$, $\alpha=(\alpha_1,\ldots,\alpha_k)\in (\NN^*)^k$  (non-zero components) with $|\alpha|=s$, 
and
\begin{equation}
\label{eq: 4 plus}
\alpha_1,\ldots,\alpha_k\leq s-2.\end{equation}
If, in $I_4$,  $k=1$ were allowed and \eqref{eq: 4 plus} omitted, the additional integrals would already be considered in $I_1$, $I_2$ and $I_3$.
Indeed observe that if $\alpha_p=s-1$ for some $p\in\{1,\ldots,k\}$, then $k=2$ and this case is included in $I_2$ and $I_3$.
If $\alpha_p=s$ for some $p\in\{1,\ldots,k\}$, then $k=1$ and this case is included in $I_1$. 
If $k\geq 3$ and $p\in\{1,\ldots,k\}$, then $s=\alpha_1+\ldots+\alpha_k\geq (k-1)+\alpha_p$ and necessarily $\alpha_p\leq s-k+1\leq s-2$ as required by \eqref{eq: 4 plus}.

The same comments hold for $I_6$ and $I_7$ as for $I_2$ and $I_3$.

In $I_8$, $a_{i,\mu,\nu,\alpha,\beta}$ is a multiple of $\partial^{(0,\mu,\nu_1,\nu_2)}a_i$ with $\mu\geq 1$,  $|\nu|=\nu_1+\nu_2=k\geq 1$, $\alpha=(\alpha_1,\ldots,\alpha_k)\in  (\NN^*)^k$ (non-zero components) and $\mu+|\alpha|=s$.

In what follows we shall assume that $s\geq 5$ and $r=||F||_{H^{s}(\sP)}$ small (but at most $\leq 1$).
See its use in \eqref{eq: W' petit consequence}.
We shall rely on the estimate from below \eqref{eq: equation coercivite}
and the symmetry property \eqref{eq: symmetry} to prove \eqref{eq: conclusion} (the main result of this subsection).

\subsubsection{The integral $I_1$}
For $F$ in $H^{s+2}_{0,per}(\sP)$,
$$I_1
\stackrel{\eqref{eq: equation coercivite}}\geq  C_{lc}  ||\partial_{x_2}^sF||_{L^2(\sP)}^2+  C_{lc}  ||\partial_{x_1}\partial_{x_2}^sF||_{L^2(\sP)}^2
$$
by the "semi-coercivity" property \eqref{eq: equation coercivite}
(there is a loss of derivatives of $F$ in the right-hand side of \eqref{eq: equation coercivite}).
\vspace{2mm}

\noindent
The aim of what follows is to ensure that $I_1$ (bounded from below by the square of a semi-norm, hence "almost" in $r^2$) "almost" controls  $I_2$, $I_3$, $I_4$ (we shall see that they are of order $o(r^2)$).
\vspace{2mm}

\subsubsection{The integrals $I_2$ and $I_3$ (inspired by Kohn-Nirenberg \cite{KoNi})}
\label{subsection: KoNi}
$$
I_3+I_2
=\int_{\sP}s\sum_{i=1}^2\sum_{q=1}^2(\partial_{z_i}\partial_{z_q}a_i)\Big(\partial_{x_2}^{s-1}\partial_{x_{i}}F\Big)\Big(\partial_{x_2}\partial_{x_{q}}F\Big)\Big(\partial_{x_2}^s\partial_{x_{i}}F\Big)\, dx_1 dx_2
$$$$
+\int_{\sP}s\sum_{q=1}^2(\partial_{z_2}\partial_{z_q}a_1)\Big(\partial_{x_2}\partial_{x_{q}}F\Big)\Big(\Big(\partial_{x_2}^{s-1}\partial_{x_{2}}F\Big)\Big(\partial_{x_2}^s\partial_{x_{1}}F\Big)
+\Big(\partial_{x_2}^{s-1}\partial_{x_{1}}F\Big)\Big(\partial_{x_2}^s\partial_{x_{2}}F\Big)\Big)\, dx_1 dx_2
$$
thanks to the symmetry hypothesis \eqref{eq: symmetry} (it is used here and in Section \ref{section: minimizer} only).
Integrations by parts give
$$
I_3+I_2
=-\int_{\sP}\frac s 2 \sum_{i=1}^2\sum_{q=1}^2\partial_{x_2}\Big((\partial_{z_i}\partial_{z_q}a_i)\Big(\partial_{x_2}\partial_{x_{q}}F\Big)\Big)\Big(\partial_{x_2}^{s-1}\partial_{x_{i}}F\Big)^2\, dx_1 dx_2
$$$$
-\int_{\sP}s\sum_{q=1}^2\partial_{x_2}\Big((\partial_{z_2}\partial_{z_q}a_1)\Big(\partial_{x_2}\partial_{x_{q}}F\Big)\Big)\Big(\partial_{x_2}^{s-1}\partial_{x_{2}}F\Big)\Big(\partial_{x_2}^{s-1}\partial_{x_{1}}F\Big)\, dx_1dx_2
$$
where we have ensured that
\begin{equation}
\label{eq: W' petit consequence}
||\partial_{z_p}\partial_{z_q}a_i||_{L^\infty(\sP)}=O(1),~~||\partial_{x_2}(\partial_{z_p}\partial_{z_q}a_i)||_{L^\infty}=O(1) ~\text{ and } ~ ||F'||_{W^{2,\infty}(\sP)}=O(r) ~\text{ as }~ s\geq 5 
\end{equation}
(thanks to the Sobolev injection $H^{s}_{0,per}(\sP)\subset W^{3,\infty}(\sP)$ if $s\geq 5$),
so that
$$I_3+I_2=O(r)||F||_{H^{s}(\sP)}^2\,.$$

\subsubsection{The integral $I_4$}
We have
$$\int_{\sP}\sum_{i,k,\alpha,\beta}a_{i,\alpha,\beta}(x,\nabla F)\Big(\partial_{x_2}^{\alpha_1}\partial_{x_{\beta_1}}F\Big)\ldots\Big(\partial_{x_2}^{\alpha_k}\partial_{x_{\beta_k}}F\Big)\Big(\partial_{x_2}^s\partial_{x_{i}}F\Big)\, dx_1 dx_2
$$$$
=-\int_{\sP}\sum_{i,k,\alpha,\beta}\partial_{x_2}\Bigg(a_{i,\alpha,\beta}(x,\nabla F)\Big(\partial_{x_2}^{\alpha_1}\partial_{x_{\beta_1}}F\Big)\ldots\Big(\partial_{x_2}^{\alpha_k}\partial_{x_{\beta_k}}F\Big)\Bigg)\Big(\partial_{x_2}^{s-1}\partial_{x_{i}}F\Big)\, dx_1 dx_2
$$$$
=-\int_{\sP}\sum_{i,k,\alpha,\beta}\partial_{x_2}\Big(a_{i,\alpha,\beta}(x,\nabla F)\Big)\Big(\partial_{x_2}^{\alpha_1}\partial_{x_{\beta_1}}F\Big)\ldots\Big(\partial_{x_2}^{\alpha_k}\partial_{x_{\beta_k}}F\Big)\Big(\partial_{x_2}^{s-1}\partial_{x_{i}}F\Big)\, dx_1dx_2
$$$$
-\int_{\sP}\sum_{i,k,\alpha,\beta}a_{i,\alpha,\beta}\Big(\partial_{x_2}^{\alpha_1+1}\partial_{x_{\beta_1}}F\Big)\ldots\Big(\partial_{x_2}^{\alpha_k}\partial_{x_{\beta_k}}F\Big)\Big(\partial_{x_2}^{s-1}\partial_{x_{i}}F\Big)\, dx_1dx_2
$$$$
-\ldots-\int_{\sP}\sum_{i,k,\alpha,\beta}a_{i,\alpha,\beta}\Big(\partial_{x_2}^{\alpha_1}\partial_{x_{\beta_1}}F\Big)\ldots\Big(\partial_{x_2}^{\alpha_k+1}\partial_{x_{\beta_k}}F\Big)\Big(\partial_{x_2}^{s-1}\partial_{x_{i}}F\Big)\, dx_1 dx_2\,.
$$
Thanks to \eqref{eq: Kato bis} applied to second-order partial derivatives of $F$, one gets when $k\geq 3$, as $|\alpha|=s$ and $\alpha_1,\ldots,\alpha_k\geq 1$, that the corresponding terms in $I_4$ are of the form
$$O(1)||F''||_{L^{\infty}(\sP)}^{k-1}||F''||_{H^{s-k}(\sP)}||F||_{H^{s}(\sP)}
+O(1)||F''||_{L^{\infty}(\sP)}^{k-1}||F''||_{H^{s-k+1}(\sP)}||F||_{H^{s}(\sP)}
$$$$
=O(r^k)||F||_{H^{s}(\sP)}
$$
where the first $O(1)$ is obtained as in \eqref{eq: W' petit consequence}.\footnote{When estimating the first sum, $a_{i,\alpha,\beta}$,  which is a multiple of a partial derivative of $a_i$ of order $k\leq s$ with respect to $z_1,z_2$, is differentiated once again, and thus at most $s+1$ times.}
When $k=2$, \eqref{eq: 4 plus} ensures that $\alpha_1,\alpha_2\geq 2$ and \eqref{eq: Kato bis} (applied to third-order partial derivatives of $F$) gives this time that the corresponding terms in $I_4$ are of the form
$$O(1)||F'''||_{L^{\infty}(\sP)}||F'''||_{H^{s-2k}(\sP)}||F||_{H^{s}(\sP)}
+O(1)||F'''||_{L^{\infty}(\sP)}||F'''||_{H^{s-2k+1}(\sP)}||F||_{H^{s}(\sP)}
$$$$
=O(r^2)||F||_{H^{s}(\sP)}
$$
by \eqref{eq: W' petit consequence}.

\subsubsection{The integral $I_5$}
We have, by an integration by parts,
$$\left|\sum_{i=1}^2\int_{\sP}(\partial_{x_2}^sa_i)\Big(\partial_{x_2}^s\partial_{x_i}F\Big)\, dx_1 dx_2\right|
\leq  ||\partial_{x_2}^{s+1}a(x,\nabla F)||_{L^2(\sP)}||\partial_{x_2}^{s-1}\nabla F||_{L^2(\sP)}\,.$$
Hence we shall require a smallness assumption on $\left\|\max_{|z|\leq \rho}|\partial_{x_2}^{s+1}a(x,z)|\,\right\|_{L^2(\sP)}$.

\subsubsection{The integrals $I_6$ and $I_7$}
As for $I_2$ and $I_3$, one gets after an integration by part
$$ |I_6 + I_7|
\leq \operatorname{Const}\max_{|\nu|=1}\left\|\partial_{x_2}\left((\partial^{(0,1,\nu_1,\nu_2)}a)(x,\nabla F)\right)\right\|_{L^\infty(\sP)} ||\partial_{x_2}^{s-1}\nabla F||_{L^2(\sP)}^2\,.$$
Hence we shall require a smallness assumption on $\left\|\max_{|z|\leq \rho}|\partial_{x_2}^{2}\nabla_z a(x,z)|\,\right\|_{L^\infty(\sP)}$ (and use that $||F''||_{L^\infty(\sP)}=O(r)$ as $s\geq 5$).

\subsubsection{The integral $I_8$}
\label{subsection: an estimate on an integral}
For $k\geq 2$, one gets $1\leq \mu\leq s-2$ and, by an integration by parts,
$$\left|\int_{\sP}\partial^{(0,\mu,\nu_1,\nu_2)}a_i(x,\nabla F)\Big(\partial_{x_2}^{\alpha_1}\partial_{x_{\beta_1}}F\Big)\ldots\Big(\partial_{x_2}^{\alpha_k}\partial_{x_{\beta_k}}F\Big)\Big(\partial_{x_2}^s\partial_{x_i}F\Big)\, dx_1 dx_2\right|
$$$$
\leq \operatorname{Const}||\partial^{(0,\mu,\nu_1,\nu_2)}a_i(x,\nabla F)||_{L^\infty(\sP)}||F''||_{H^{s+1-\mu-k}_{per}(\sP)}||F''||_{L^\infty(\sP)}^{k-1}||\partial_{x_2}^{s-1}\nabla F||_{L^2(\sP)}
$$$$
+ \operatorname{Const}||\partial^{(0,\mu+1,\nu_1,\nu_2)}a_i(x,\nabla F)||_{L^\infty(\sP)}||F''||_{H^{s-\mu-k}_{per}(\sP)}||F''||_{L^\infty(\sP)}^{k-1}
||\partial_{x_2}^{s-1}\nabla F||_{L^2(\sP)}
$$$$
+ \operatorname{Const}||\partial^{(0,\mu,1+\nu_1,\nu_2)}a_i(x,\nabla F)||_{L^\infty(\sP)}||F''||_{H^{s-\mu-k}_{per}(\sP)}||F''||_{L^\infty(\sP)}^{k}
||\partial_{x_2}^{s-1}\nabla F||_{L^2(\sP)}
$$$$
+ \operatorname{Const}||\partial^{(0,\mu,\nu_1,1+\nu_2)}a_i(x,\nabla F)||_{L^\infty(\sP)}||F''||_{H^{s-\mu-k}_{per}(\sP)}||F''||_{L^\infty(\sP)}^{k}
||\partial_{x_2}^{s-1}\nabla F||_{L^2(\sP)}
\,.$$

For $k=1$, we get $\mu=s-\alpha_1$ with $1\leq \alpha_1\leq s-2$, and therefore an integration by parts gives
$$\left|\int_{\sP}\partial^{(0,\mu,\nu_1,\nu_2)}a_i(x,\nabla F)\Big(\partial_{x_2}^{\alpha_1}\partial_{x_{\beta_1}}F\Big)\Big(\partial_{x_2}^s\partial_{x_i}F\Big)\, dx_1 dx_2\right|
$$$$
\leq \operatorname{Const}||\partial^{(0,s-\alpha_1,\nu_1,\nu_2)}a_i(x,\nabla F)\partial_{x_2}^{\alpha_1+1}\nabla F||_{L^2(\sP)} ||\partial_{x_2}^{s-1}\nabla F||_{L^2(\sP)}
$$$$
+ \operatorname{Const}||\partial^{(0,s-\alpha_1+1,\nu_1,\nu_2)}a_i(x,\nabla F)\partial_{x_2}^{\alpha_1}\nabla F||_{L^2(\sP)}||\partial_{x_2}^{s-1}\nabla F||_{L^2(\sP)}
 +O(r^2)||\partial_{x_2}^{s-1}\nabla F||_{L^2(\sP)} 
$$$$
\leq  \sum_{\tilde \alpha_1=1}^{s-1}\operatorname{Const}||\partial^{(0,s-\tilde \alpha_1+1,\nu_1,\nu_2)}a_i(x,\nabla F)||_{L^\infty(\sP)}||\partial_{x_2}^{\tilde \alpha_1}\nabla F||_{L^2(\sP)} ||\partial_{x_2}^{s-1}\nabla F||_{L^2(\sP)}
$$$$
+O(r^2)||\partial_{x_2}^{s-1}\nabla F||_{L^2(\sP)} 
$$$$
\leq \operatorname{Const}\left(||F||_{H^{s-1}(\sP)}
+\left\|\max_{|z|\leq \rho}|\partial_{x_2}^{2}\nabla_z a(x,z)|\,\right\|_{L^\infty(\sP)}||\partial_{x_2}^{s-1}\nabla F||_{L^2(\sP)}+r^2\right)||\partial_{x_2}^{s-1}\nabla F||_{L^2(\sP)}
$$
with $\nu_1+\nu_2=1$.
Hence  a smallness assumption on $\left\|\max_{|z|\leq \rho}|\partial_{x_2}^{2}\nabla_z a(x,z)|\,\right\|_{L^\infty(\sP)}$ is again useful.

\subsubsection{End of the proof of \eqref{eq: conclusion}}
The various estimates obtained on the integrals $I_1$ to $I_8$ imply inequality \eqref{eq: conclusion}.
The terms needing more careful estimates because of the $x$-dependence appear in $I_5$--$I_8$.

\subsection{Appendix}
This Appendix contains some simple but useful intermediate results stated for the toy problem.
First, we state a technical lemma concerning Sobolev spaces
for domains in $\mathbb R^2$ that are bounded in the first variable and periodic in the second. The proof relies on Gagliardo-Nirenberg interpolation estimates in bounded domains. Inequality \eqref{eq: Kato bis} below is an adaptation of the very useful Lemma A1 in \cite{Kato} which dealt with functions defined on the $k$-torus. Domains in $\mathbb R^3$ that are bounded in the first variable and periodic in the two last ones can be dealt with in the same way.

\begin{lemma}
Let the integers $m_1,m_2\geq 0$  be such that $m=m_1+m_2$ with $m_1<m$.
Given $\epsilon>0$, for all $u\in C^\infty([0,1]\times \RR)$ that is $1$-periodic in $x_2$
\begin{equation}
\label{eq: a Sobolev interpolation}
||\partial_{x_1}^{m_1}\partial_{x_2}^{m_2}u||_{L^2(\sP)}^2
	\leq \epsilon||\partial_{x_1}^m u||_{L^2(\sP)}^2
+C_\epsilon||\partial_{x_2}^mu||_{L^2(\sP)}^2
\end{equation}
for some constant $C_\epsilon>0$ that depends on $\epsilon$.
In addition, let $\alpha_j\in \mathbb N^2$, $j=1,2,\ldots,k$, be multi-indices with $|\alpha_1|+\ldots+|\alpha_k|=p$ ($|\alpha_j|$ being the sum of the two components of $\alpha_j$). 
For all $u_1,\ldots,u_k\in H^p(\sP)$,
\begin{equation}
\label{eq: Kato bis}
||(\partial^{\alpha_1}u_1)\cdots(\partial^{\alpha_k}u_k)||_{L^2(\sP)}
\leq c ||u_1||_{H^p(\sP)}||u_2||_{L^\infty(\sP)}\cdots ||u_k||_{L^\infty(\sP)} +\text{(cycl)}\,
\end{equation}
where $c>0$ is some constant and "(cycl)" means a summation on the $k$ cyclic permutations of the indices $1,\ldots,k$ (if $k\geq 2$).
\end{lemma}
\begin{proof} Using Fourier's series in $x_2$,
$$u(x_1,x_2)=\sum_{k_2\in\mathbb Z} f_{k_2}(x_1)e^{2i\pi k_2 x_2}$$
with $f_{-k_2}=\overline{f_{k_2}}$,
we get
$$||\partial_{x_1}^{m_1}\partial_{x_2}^{m_2}u||_{L^2(\sP)}^2
=\sum_{k_2\in\mathbb Z} (2\pi k_2)^{2m_2}||f_{k_2}^{(m_1)}||_{L^2(0,1)}^{2},
$$
$$||\partial_{x_1}^{m}u||_{L^2(\sP)}^2
=\sum_{k_2\in\mathbb Z} ||f_{k_2}^{(m)}||_{L^2(0,1)}^{2}
$$
and
$$
||\partial_{x_2}^{m}u||_{L^2(\sP)}^2
=\sum_{k_2\in\mathbb Z\setminus\{0\}} (2\pi k_2)^{2m}||f_{k_2}||_{L^2(0,1)}^{2}.
$$
From \cite[Theorem 5.1]{Adams-Fournier} (see also Theorems 7.40 and Theorems 7.41 in \cite{Le}), for each $\epsilon>0$ there is $K_\epsilon\geq 1$ such that for all $f\in H^{(m)}(0,1)$ and all $0<\delta\leq \epsilon$,
\begin{equation}
\label{eq: 1 d interpolation bis}
||f^{(m_1)}||_{L^2(0,1)}^{2}\leq K_\epsilon\left(\delta ||f^{(m)}||_{L^2(0,1)}^{2}+
\delta^{-m_1/(m-m_1)}||f||_{L^2(0,1)}^{2}\right).\end{equation}
Applying this estimate to $f=f_{k_2}$ and $\delta=(2\pi k_2)^{-2m_2}\epsilon/K_\epsilon$, we get, for each $k_2\in\mathbb Z\setminus\{0\}$,
$$
(2\pi k_2)^{2m_2}||f_{k_2}^{(m_1)}||_{L^2(0,1)}^{2}\leq \epsilon ||f_{k_2}^{(m)}||_{L^2(0,1)}^{2}+(2\pi k_2)^{2m}K_\epsilon^{1+\frac{m_1}{m_2}}\, \epsilon^{-m_1/m_2}||f_{k_2}||_{L^2(0,1)}^{2}.
$$
This inequality is also true for $k_2=0$, since in this case the left-hand side is zero. Thus, summing up over all values of $k_2\in \mathbb Z$, and taking $C_\varepsilon:=K_\epsilon^{1+\frac{m_1}{m_2}}\, \epsilon^{-m_1/m_2}$, we obtain \eqref{eq: a Sobolev interpolation}.

$$$$

We now prove \eqref{eq: Kato bis}. Note that this proof is essentially contained in \cite{Kato}, we only repeat it here for the reader's convenience. By H\"{o}lder's inequality,
$$
||(\partial^{\alpha_1}u_1)\cdots(\partial^{\alpha_k}u_k)||_{L^2(\sP)}
\leq ||\partial^{\alpha_1}u_1||_{L^{\frac{2p}{|\alpha_1|}}(\sP)}\cdots
||\partial^{\alpha_k}u_k||_{L^{\frac{2p}{|\alpha_k|}}(\sP)},
$$
 where $L^{\frac{2p}{|\alpha_i|}}(\sP)=L^\infty(\sP)$ if $|\alpha_i|=0$.
By the Gagliardo-Nirenberg inequality in bounded domains (see \cite{Nirenberg}), there is a positive constant $C$ such that for each $1\leq i\leq k$,
$$
||\partial^{\alpha_i}u_i||_{L^{\frac{2p}{|\alpha_i|}}(\sP)}\leq C ||u_i||^{\frac{|\alpha_i|}{p}}_{H^{p}(\sP)}||u_i||_{L^{\infty}(\sP)}^{1-\frac{|\alpha_i|}{p}},
$$
so, taking the product of these inequalities and using the identity $1-\frac{|\alpha_i|}{p}=\sum_{j\neq i} \frac{|\alpha_j|}{p}$ to rearrange the terms of the right-hand side,
$$
\prod_{i=1}^k||\partial^{\alpha_i}u_i||_{L^{\frac{2p}{|\alpha_i|}}(\sP)}\leq C^k\prod_{i=1}^k
\left(||u_i||_{H^{p}(\sP)}\prod_{j\neq i}||u_j||_{L^{\infty}(\sP)}\right)^{\frac{|\alpha_i|}{p}}.$$
Thus, using the inequality of arithmetic and geometric means,
$$\prod_{i=1}^k||\partial^{\alpha_i}u_i||_{L^{\frac{2p}{|\alpha_i|}}(\sP)}\leq C^k\sum_{i=1}^k
\frac{|\alpha_i|}{p}||u_i||_{H^{p}(\sP)}\prod_{j\neq i}||u_j||_{L^{\infty}(\sP)},
$$
hence \eqref{eq: Kato bis} with $c:=C^k$.

\end{proof}

\begin{lemma}
\label{lemme: controle a}
$\partial_{x_2}^{s-1}(a(x,\nabla w))\in L^2(\sP)^2$ for all $w\in Y$, and 
$$
w_n\rightharpoonup w \text{ in }Y~~\Rightarrow~~
\partial_{x_2}^{s-1}(a(x,\nabla w_n))\rightharpoonup\partial_{x_2}^{s-1}(a(x,\nabla w))\text{ in }L^2(\sP)^2\,.
$$
\end{lemma}
\begin{proof}
Concerning compositions of functions, if $g:\mathbb R^p\rightarrow \mathbb R$ is of class $C^1$ and if $f=(f_1,\ldots, f_p):(0,1)^2\rightarrow \mathbb R^p$ is in $H^1((0,1)^2)\cap L^\infty ((0,1)^2)$, then $g\circ f$ is in $H^1((0,1)^2)$ and, for all $j\in\{1,2\}$,
$$\frac{\partial(g\circ f)}{\partial x_j}(x)=\sum_{k=1}^{p}\frac{\partial g}{\partial y_k}(f(x))\frac{\partial f_k}{\partial x_j}(x).$$
When $p=1$ and $g$ has a bounded derivative, see e.g.  Proposition 9.5 in \cite{Brezis} applied to the open set $(0,1)^2$.
This formula can be applied for example when $p>2$, $f_1(x)=x_1$ and $f_2(x)=x_2$.

One gets, for $w\in Y$ and $i\in\{1,2\}$, that $\partial_{x_2}^{s-1}a_i(x,\nabla w)$ is a finite sum of terms of the form
$$g(x,\nabla w)\partial_{x_2}^{\ell_1}\tilde w_1\cdots \partial_{x_2}^{\ell_k}\tilde w_k\,,$$
where
$$k\in\{0,\ldots,s-1\},~~\ell_1,\ldots, \ell_k\geq 1,~~\ell_1+\ldots+\ell_k\leq s-1,~~\tilde w_1,\ldots,\tilde w_k\in\{\partial_{x_1}w,\partial_{x_2}w\}$$
and $g$ is a partial derivative of order $\geq k$ of $a_i$ (there could be many terms of type $k$).
Apply Lemma A1 in \cite{Kato} for each fixed $x_1$ and then integrate over $x_1$ to obtain that
$$||g(x,\nabla w)\partial_{x_2}^{\ell_1}\tilde w_1\cdots \partial_{x_2}^{\ell_k}\tilde w_k||_{L^2(\sP)}^2\leq \operatorname{Const}||g(x,\nabla w)||_{L^\infty}^2||\nabla w||_{L^\infty}^{2(k-1)}\big(\sum_{\ell=0}^{s-1}||\partial_{x_2}^{\ell}\nabla w||^2_{L^2}\Big)\,.$$
When $k=0$, this becomes simply $||g(x,\nabla w)||_{L^2(\sP)}^2 \leq ||g(x,\nabla w)||_{L^\infty}^2$.

As $s\geq 5,$ the Sobolev injection $H^s_{0,per}(\sP)\subset C_{per}^{s-2}(\overline \sP)$ implies that $\nabla w \in L^\infty(\sP)^2$.
Moreover the hypothesis $w\in Y$
ensures that $\partial_{x_2}^{\ell}\nabla w\in L^{2}(\sP)^2\,$ for $1\leq \ell\leq s-1$.
This proves that $\partial_{x_2}^{s-1}(a(x,\nabla w))\in L^2(\sP)^2$ for all $w\in Y$.
The injection $H^s_{0,per}(\sP)\subset C_{per}^{s-2}(\overline \sP)$ being compact, the weak convergence $w_n\rightharpoonup w$ in $Y$ implies that the following limits are uniform:
$$\partial_{x_2}^{\ell}\nabla w_n\rightarrow \partial_{x_2}^{\ell}\nabla w
~\text{ for all  }~
\ell\in\{0,\ldots,s-3\}$$
In addition
$$\partial_{x_2}^{\ell}\nabla w_n\rightharpoonup \partial_{x_2}^{\ell}\nabla w
~\text{ for  }~
\ell\in\{s-2,s-1\}.$$
Hence
$$g(x,\nabla w_n)\partial_{x_2}^{\ell_1}(\tilde w_1)_n\cdots \partial_{x_2}^{\ell_k}(\tilde w_k)_n \rightharpoonup
g(x,\nabla w)\partial_{x_2}^{\ell_1} \tilde w_1 \cdots \partial_{x_2}^{\ell_k}\tilde w_k,
~~ (\tilde w_1)_n,\ldots,(\tilde w_k)_n\in\{\partial_{x_1}w_n,\partial_{x_2}w_n\},$$
because at most one of $\ell_1,\ldots,\ell_k$ belongs to $\{s-2,s-1\}$.
This proves the second statement.
\end{proof}

\begin{corollary}
\label{cor: controle a}
$\partial_{x_2}^{s}(a(x,\nabla w))\in L^2(\sP)^2$ for $w\in V$.
\end{corollary}

\begin{lemma}
\label{lemme: controle div a}
$\partial_{x_1}^{s-2}\operatorname{div}(a(x,\nabla w))\in L^2(\sP)$ for all $w\in Y$, and
$$
w_n\rightharpoonup w \text{ in }Y~~\Rightarrow~~
\partial_{x_1}^{s-2}\operatorname{div}(a(x,\nabla w_n))\rightharpoonup\partial_{x_1}^{s-2}\operatorname{div}(a(x,\nabla w))\text{ in }L^2(\sP)\,.
$$
\end{lemma}
\begin{proof}
Similar to the proof of Lemma \ref{lemme: controle a}, with the help of \eqref{eq: dvpt}.
\end{proof}

\begin{lemma}
\label{eq: mixed derivatives}
Let the integers $m_1,m_2,m \geq 0$  be such that $0<m_1+m_2<m$.
Given $\epsilon>0$, for all $u\in C^\infty([0,1]\times \RR)$ that is $1$-periodic in $x_2$
\begin{equation*}
||\partial_{x_1}^{m_1}\partial_{x_2}^{m_2}u||_{L^2(\sP)}^2
\leq  \epsilon \Big(
||\partial_{x_1}^m u||_{L^2(\sP)}^2+||\partial_{x_2}^mu||_{L^2(\sP)}^2\Big)
+C_\epsilon||u||_{L^2(\sP)}^2
\end{equation*}
for some constant $C_\epsilon>0$ that depends on $\epsilon$.
\end{lemma}

\begin{proof}
Given two integers $0<\sigma<s$, the interpolation inequality
\begin{equation}
\label{eq: interpolation in x_2}
||\partial_{x_2}^{\sigma}f||_{L^2(\sP)}^2\leq ||f||_{L^2}^{2(s-\sigma)/s}||\partial_{x_2}^{s}f||_{L^2}^{2\sigma/s}
\end{equation}
holds for all $f\in H_{0,per}^s(\sP)$.
Indeed it suffices to interpolate $|2\pi k_2|^{2\sigma}$ with $0<\sigma<s$ between the two functions $1$ and $|2\pi k_2|^{2s}$, and to use Fourier's series in $x_2$, which gives that
$|2\pi k_2|^{2\sigma}= 1^{(s-\sigma)/s}(|2\pi k_2|^{2s})^{\sigma/s}$ and \eqref{eq: interpolation in x_2}

Now, by \eqref{eq: a Sobolev interpolation} where $\epsilon=1$ and $m$ is replaced by $m_1+m_2$,
there is some constant $C>0$ such that
$$||\partial_{x_1}^{m_1}\partial_{x_2}^{m_2}u||_{L^2(\sP)}^2
\leq C\Big(|| u||_{L^2(\sP)}^2+ ||\partial_{x_1}^{m_1+m_2} u||_{L^2(\sP)}^2+||\partial_{x_2}^{m_1+m_2}u||_{L^2(\sP)}^2\Big)$$
and the statement follows from \eqref{eq: 1 d interpolation bis}(with $\delta=\frac{\epsilon}{K_\epsilon}$) and \eqref{eq: interpolation in x_2}. 
\end{proof}

\section{The hydrodynamic problem} \label{section: hydro}

\subsection{Preliminary remarks}
Let $H$ be like in Theorem \ref{thm: hydrodynamics} and
let $(f,g)$ in equation \eqref{eq: deux} be replaced by $(f+F,g+G)$, where both pairs $(f,g)$ and $(f+F,g+G)$ are admissible. Equation \eqref{eq: deux} can be written for the unknown pair $(F,G)$ as 
\begin{equation}\label{eq: deux bis}
\left(\begin{array}{c}
-\Div(\nabla (g+G) \times (\nabla (f+F)\times \nabla(g+G)))+
\partial_f H(f+F,g+G)
\\
-\Div((\nabla (f+F)\times \nabla (g+G))\times \nabla (f+F))
+\partial_g H(f+F,g+G)
\end{array}\right)=0,
\end{equation}
where
$(F,G)\in C^2(\overline D)$ is $(P_y,P_z)$-periodic in $y$ and $z$ and vanishes on $\partial D$.
If
$\widetilde {\operatorname{div}}$ on a vector field $(x,y,z)\rightarrow W\in \mathbb R^6$ denotes
$\Big(\operatorname{div}(W_1,W_2,W_3),\operatorname{div}(W_4,W_5,W_6)\Big)\in \mathbb R^2$,
then the problem can be written
$$-\widetilde{\operatorname{div}}\Big(a(\nabla (f+F),\nabla (g+G))\Big)+(\nabla H)(f+F,g+G)=0,$$
where $a=(a_1,\ldots,a_6)=(a_{1,2,3},a_{4,5,6})\in \mathbb R^6$ with
$$a_{1,2,3}(w_1,\ldots,w_6)=(w_4,w_5,w_6) \times ((w_1,w_2,w_3)\times (w_4,w_5,w_6))
$$$$
\stackrel{\text{Lagrange's formula}}=||(w_4,w_5,w_6)||^2(w_1,w_2,w_3)-(w_1w_4+w_2w_5+w_3w_6)(w_4,w_5,w_6)
$$
\begin{equation}
\label{eq: developed formula}
=(w_5^2w_1+w_6^2w_1-w_2w_4w_5-w_3w_4w_6,w_4^2w_2+w_6^2w_2-w_1w_4w_5-w_3w_5w_6,w_4^2w_3+w_5^2w_3-w_1w_4w_6-w_2w_5w_6 )
\end{equation}
and
$$a_{4,5,6}(w_1,\ldots,w_6)=((w_1,w_2,w_3)\times (w_4,w_5,w_6))\times (w_1,w_2,w_3)
$$$$
=(w_1,w_2,w_3) \times ((w_4,w_5,w_6)\times (w_1,w_2,w_3)).
$$
In weak form, this can be written as follows:
for all pairs $(\delta F,\delta G)\in H^5_{0,per}(\sP)$
$$\int_{\sP} a(\nabla (f+F),\nabla (g+G))\cdot (\nabla \delta F,\nabla \delta G)\,dx\,dy\,dz
+\int_{\sP}(\nabla H)(f+F,g+G)\cdot(\delta F,\delta G)\,dx\,dy\,dz=0.$$
The directional derivative (with respect to $(F,G)$) at $(F,G)=(0,0)$ in the direction $(F,G)$ of the left-hand side of equation \eqref{eq: deux bis} can therefore also be written in a weak form
$B_{(f,g)}((F,G),(\delta F,\delta G))$ for all such pairs $(\delta F,\delta G)$,
where the bilinear form $B_{(f,g)}$ is given by
\begin{align*}
&
B_{(f,g)}((F,G),(\delta F,\delta G))
=\int_{\sP}\Big\{ 
\sum_{i,j=1}^6(\nabla \delta F,\nabla \delta G)((\partial_{w_j}a_i)(\nabla f,\nabla g))(\nabla F,\nabla  G)^T
\\&\qquad\qquad+
\partial_f^2 H(f,g)F\delta F+\partial_f \partial_g H(f,g) (F\delta G+G\delta F) + \partial_g^2 H(f,g) G \delta G
\Big\}\,dx\,dy\,dz
\end{align*}
with $(\nabla F,\nabla G)$ and $(\nabla \delta F,\nabla \delta G)$ written as row vectors, the index $i$ being for the rows and the index $j$ for the columns of the matrix $(\partial_{w_j}a_i(\nabla f,\nabla g))$.
Observe that \eqref{eq: deux bis} can be seen as the equation 
for critical points (in the sense that all directional derivatives at $(F,G)$  vanish) of the integral functional
$$
(F,G)\rightarrow\int_{\sP}\Big\{ \frac12|\nabla( f+F)\times \nabla (g+G)|^2+H(f+F,g+G)\Big\}\,dx\,dy\,dz,
$$
and the bilinear form $B$ is related to the second Fr\'echet derivative of the integral functional at $(F,G)=(0,0)$.
This intuitively explains why the bilinear form $B_{(f,g)}$ is symmetric, in analogy with \eqref{eq: symmetry}, that is, the Jacobi matrix $J_a(w_1,\ldots,w_6)$ of $a$ at $(w_1,\ldots,w_6)$ is symmetric:
$$\partial_{w_j}a_i=\partial_{w_i}a_j\,,~~i,j\in\{1,\ldots,6\}$$
(and this can be checked explicitly).
To sum up, the hydrodynamic problem is of a form analogous to \eqref{eq: the toy equation}:
the aim is to find a couple of functions $(F,G)\in H^5_{0,per}(\sP)$ such that
\begin{equation}
\label{eq: deux ter}
-\widetilde{\operatorname{div}}\Big(a(\nabla (\widetilde f_0+F),\nabla (\widetilde g_0+G))\Big)+(\nabla H)(\widetilde f_0+F,\widetilde g_0+G)=0.
\end{equation}

The next proposition will lead to a lower estimate analogous to \eqref{eq: consequence of semi-coer} on the integral \eqref{eq: for Delta4 F G} below.
Its proof is similar to an estimate in the first step of the proof of Theorem 2.1 in \cite{BuWa}.

\begin{prop}
\label{prop: simple inequality}
Assume that $(f,g)\in C^1(\overline D)$ is such that 
$$\nabla_{(y,z)} f=(\partial_y f,\partial_z f)\in C(\overline D,\mathbb R^2)~\text{ and }~\nabla_{(y,z)} g$$ 
are $(P_y,P_z)$-periodic in $y$ and $z$,  and 
$|\nabla_{(y,z)} f|^2+|\nabla_{(x,y)} g|^2$
never vanishes (this is the case if
the first component $v_1$ of
$v = \nabla f \times \nabla g$
never vanishes).
Then for all $(F,G)\in L^2_{loc}(D) $  that is $(P_y,P_z)$-periodic in $y$ and $z$, we have
$$\int_{\sP} \Big( (\partial_{w_1}a_1)(\nabla f,\nabla g)F^2+(\partial_{w_4}a_1)FG+(\partial_{w_1}a_4)FG+(\partial_{w_4}a_4)G^2\Big)\, dx\, dy\, dz$$
$$
\stackrel{\eqref{eq: developed formula}}=\int_{\sP}\Big(|\nabla_{(y,z)} g|^2 F^2+|\nabla_{(y,z)} f|^2G^2-2(\nabla_{(y,z)} f\innpr \nabla_{(y,z)} g)FG\Big)\, dx\, dy\, dz 
$$$$
\geq\int_{\sP}\frac{v_1^2}{|\nabla_{(y,z)} f|^2+|\nabla_{(y,z)} g|^2}\Big(F^2+G^2\Big)\, dx\, dy\, dz
$$
with
$$v_1^2=(\partial_yf\partial_zg-\partial_z f\partial_yg)^2
=|\nabla_{(y,z)} f|^2|\nabla_{(y,z)} g|^2-|\nabla_{(y,z)} f\cdot \nabla_{(y,z)} g|^2\,.$$
\end{prop}

\begin{proof}
$$\int_{\sP}\Big(|\nabla_{(y,z)} g|^2 F^2+|\nabla_{(y,z)} f|^2G^2-2(\nabla_{(y,z)} f\cdot \nabla_{(y,z)} g)FG\Big)\, dx\, dy\, dz 
$$
$$\geq\int_{\sP}
\frac{|\nabla_{(y,z)} f|^2+|\nabla_{(y,z)} g|^2-\sqrt{(|\nabla_{(y,z)} f|^2+|\nabla_{(y,z)} g|^2)^2-4v_1^2}}{2} \left(F^2+G^2\right)\,dx\,dy\,dz$$
because the eigenvalues of 
$$\left(\begin{array}{c c}|\nabla_{(y,z)} g|^2&- \nabla_{(y,z)} f\innpr\nabla_{(y,z)} g\\
- \nabla_{(y,z)} f\innpr\nabla_{(y,z)} g&|\nabla_{(y,z)} f|^2\end{array}\right)$$
are $\frac 1 2 \left(
|\nabla_{(y,z)} f|^2+|\nabla_{(y,z)} g|^2\pm \sqrt{(|\nabla_{(y,z)} f|^2+|\nabla_{(y,z)} g|^2)^2-4v_1^2}
\right).$
Thus
$$\dots
\geq\int_{\sP}
\frac{2v_1^2}{|\nabla_{(y,z)} f|^2+|\nabla_{(y,z)} g|^2+\sqrt{(|\nabla_{(y,z)} f|^2+|\nabla_{(y,z)} g|^2)^2-4v_1^2}} \left(F^2+G^2\right)\,dx\,dy\,dz
$$
$$
\geq\int_{\sP}
\frac{v_1^2}{|\nabla_{(y,z)} f|^2+|\nabla_{(y,z)} g|^2} \left(F^2+G^2\right)\,dx\,dy\,dz.
$$
\end{proof}

An important feature is the possibility of bounding from below the quadratic functional 
\\$\frac 1 2 B_{(f,g)}((F,G),(F,G))$.
Let us recall the following inequality in \cite{BuWa} (its statement is here slightly more complicated as a simplifying assumption has not been introduced).
It plays an analogous role as the second hypothesis in equation \eqref{eq: hyp on a}, see also \eqref{eq: equation coercivite}. However, its proof involves an integral similar to an integral of a determinant (see in \cite{BuWa} the second step of the proof of Theorem 2.1). 

\begin{thm}
\label{thm: quadratic part}
Assume that $\nabla \overline f$ and $\nabla \overline g$ are constant and
that the first component of $\overline v$ does not vanish.
Let $H\in C^2(\RR^2)$.
For $(f,g)\in C^3(\overline D)$ and $(F,G)\in H_{loc}^1(D)$ such that $\nabla f$, $\nabla g$, $F$, $G$ and $H''(f,g)$ are $(P_y,P_z)$-periodic in $y$ and $z$,  and such that $(F,G)=0$ on $\partial D$ in the usual weak sense of traces, we have
\begin{equation*}
\begin{aligned}
& B_{(f,g)}((F,G), (F,G))
\\&\geq\frac{1}{2|\nabla \overline f|^2+2|\nabla \overline g|^2}
\int_{\sP}\Big\{ 
\frac 1 {16}(v\innpr \nabla F)^2+\frac 1{16}(v\innpr \nabla G)^2+(1-O(\|v'\|_{C(\overline \sP)}))\frac {\pi^2\min_{\overline \sP}  v_1^2} {16L^2}(F^2+G^2)\Big\}\,dx\,dy\,dz 
\\&~~~~~~~~~~~ +\int_{\sP}\Big\{ 
\partial_f^2 H(f,g)F^2+2\partial_f \partial_g H(f,g)FG+ \partial_g^2 H(f,g) G^2
\Big\}\,dx\,dy\,dz
\end{aligned}
\end{equation*}
if $(\nabla f,\nabla g)$ is in some small neighborhood of
$(\nabla \overline f,\nabla \overline g)$ in  $C^2(\overline \sP)$
(independent of $H$). Here $v'$ denotes the Jacobian matrix of $v$.
	The notation $O(||v'||_{C(\overline \sP)})$ is for a continuous function whose $L^\infty$-norm is bounded by a constant times $||v'||_{C(\overline \sP)}$ (the constant being independent of $(F,G)$, $(f,g)$ and $H$).
\end{thm}
\begin{rem}
The assumption $(f,g)\in C^3(\overline D)$ will be ensured by the property $(f,g)\in H^5_{loc}(D)$.
\end{rem}

\subsection{Some estimates from below}
In order to apply Kato's abstract Theorem \ref{thm: Kato},  \eqref{eq: Kato coercivity} will be checked with the help of intermediate estimates from below proved in this subsection.
Let us consider a fixed pair $(\widetilde f_0,\widetilde g_0)\in H^7_{loc}(D)$ that is $(P_y,P_z)$-periodic in $y$ and $z$, on which we shall add assumptions as required in the statement of Theorem  \ref{thm: hydrodynamics}.
Let $r\in]0,1]$ be as small as needed,  consider the closed ball $\overline B(0,r)\subset H_{0,per}^5(\sP)$ and assume that $(F,G)$ is on the boundary $\partial B(0,r)$.
Thanks to \eqref{eq: estimates in r square}, we can also assume that
\begin{equation}\label{eq: two estimates}
||(\widetilde f_0'',\widetilde g_0'')||_{H^5(\sP)}\leq \xi r\,.
\end{equation}
Remember that the indices $0,per$ is for the periodicity conditions in $y$ and $z$ on $(F,G)$, and the Dirichlet condition $F=G=0$ for $x\in\{0,L\}$.

\subsubsection{First estimate from below}
Let us note $(x,y,z)=(x_1,x_2,x_3)$.
Arguing as in Subsection \ref{section: proof of crucial inequality} to prove \eqref{eq: conclusion},
we get (see \eqref{eq: deux bis} and \eqref{eq: deux ter})
$$\int_{\sP}\left(\frac {\partial^5}{\partial y^5}F\right)
\frac{\partial^5}{\partial y^5}\left(-\Div(\nabla (\widetilde g_0+G) \times (\nabla (\widetilde f_0+F)\times \nabla (\widetilde g_0+G)))+
\partial_f H(\widetilde f_0+F,\widetilde g_0+G)\right)
\,dx\,dy\,dz$$
$$
+\int_{\sP}\left(\frac {\partial^5}{\partial y^5}G\right)
\frac{\partial^5}{\partial y^5}\left(-\Div((\nabla (\widetilde f_0+F)\times \nabla (\widetilde g_0+G))\times \nabla (\widetilde f_0+F)))
+\partial_gH(\widetilde f_0+F,\widetilde g_0+G)\right)
\,dx\,dy\,dz$$
$$=\int_{\sP}\sum_{i=1}^3\left(\frac {\partial^5}{\partial y^5}\partial_{x_i}F\right)
\frac{\partial^5}{\partial y^5}\left(a_i(\nabla (\widetilde f_0+F),\nabla (\widetilde g_0 + G))\right)
\,dx\,dy\,dz$$
$$+\int_{\sP}\sum_{i=1}^3\left(\frac {\partial^5}{\partial y^5}\partial_{x_i}G\right)
\frac{\partial^5}{\partial y^5}\left(a_{3+i}(\nabla (\widetilde f_0+F),\nabla (\widetilde g_0 + G))\right)
\,dx\,dy\,dz\stackrel{\eqref{eq: estimates in r square}}{-}\operatorname{Const} \xi r^2$$
$$\stackrel{\text{as in Subsection \ref{section: proof of crucial inequality}}} \geq \,  B_{(\widetilde f_0+F,\widetilde g_0 +G)}((\partial_y^5F,\partial_y^5G),(\partial_y^5F,\partial_y^5G))
-\operatorname{Const} (r^3+\xi r^2 ).$$
Note that the analogue of $I_2$+$I_3$ is, with $s=5$, $(x,y,z)=(x_1,x_2,x_3)$ and $E=(\nabla F,\nabla G)\in \mathbb R^6$,
$$
\int_{\sP}s \sum_{i=1}^{6}\sum_{q=1}^6\sum_{p=1}^6(\partial_{w_p}\partial_{w_q}a_i)(\nabla (\widetilde f_0+F),\nabla (\widetilde g_0 + G))\cdot\Big(\partial_{x_2}^{s-1}E_p\Big)\Big(\partial_{x_2}E_q\Big)\Big(\partial_{x_2}^sE_i\Big)\, dx_1 dx_2 dx_3
$$
where
$\partial_{w_p}\partial_{w_q}a_i=\partial_{w_i}\partial_{w_q}a_p$,
for $p,q,i\in\{1,\ldots,6\}$.
By Theorem \ref{thm: quadratic part},
$$\ldots\geq\frac 1 {2|\nabla \overline f|^2+2|\nabla \overline g|^2} \int_{\sP}\frac {\pi^2\min_{\overline \sP}  \widetilde v_{01}^2} {32L^2}((\partial_y^5F)^2+(\partial_y^5G)^2) \, dx\, dy\, dz 
-\operatorname{Const} (r^3+\xi r^2)$$
\begin{equation}\label{eq: conclusion bis}
\geq K_1||\partial_y^5(F,G)||_{L^2(\sP)}^2-\operatorname{Const} (r^3+\xi r^2),
\end{equation}
where $\widetilde v_{01}$ is the first component of $\widetilde v_0=\nabla \widetilde f_0\times\nabla \widetilde g_0$ and $K_1>0$ is some constant.
Theorem \ref{thm: quadratic part} can be applied because,
for small enough $r\in (0,1]$, $||(\nabla \widetilde f_0+\nabla F,\nabla \widetilde g_0+\nabla G)-(\nabla \overline f,\nabla \overline g)||_{H^4(\sP)}$ is as small as needed (see \eqref{eq: estimates in r square}) and $||\big(\nabla(\widetilde f_0+F)\times\nabla(\widetilde g_0+G)\big)'||_{L^\infty(\sP)}$ is as small as needed too because of \eqref{eq: two estimates}.

Terms of the type  $\operatorname{Const} r ||(F,G)||_{H^{s-1}(\sP)}$ do not appear explicitly (see the end of Paragraph \ref{subsection: an estimate on an integral}), as here they are included in those of the type $\operatorname{Const} \xi r^2$.

\subsubsection{Second estimate from below}
In the same way,
$$\int_{\sP}\left(\frac {\partial^5}{\partial z^5}F\right)
\frac{\partial^5}{\partial z^5}\left(-\Div(\nabla (\widetilde g_0+G) \times (\nabla (\widetilde f_0+F)\times \nabla (\widetilde g_0+G)))+
\partial_f H(\widetilde f_0+F,\widetilde g_0+G)\right)
\,dx\,dy\,dz$$
$$
+\int_{\sP}\left(\frac {\partial^5}{\partial z^5}G\right)
\frac{\partial^5}{\partial z^5}\left(-\Div((\nabla (\widetilde f_0+F)\times \nabla (\widetilde g_0+G))\times \nabla (\widetilde f_0+F)))
+\partial_gH(\widetilde f_0+F,\widetilde g_0+G)\right)
\,dx\,dy\,dz$$
\begin{equation}
\label{eq: conclusion ter}
\geq K_1||\partial_z^5(F,G)||_{L^2(\sP)}^2
-\operatorname{Const}(r^3+\xi r^2).
\end{equation}

\subsubsection{Third estimate from below}
To get an inequality analogous to \eqref{eq: consequence of semi-coer},
we now study
$$
\int_\sP
-\frac{\partial^3}{\partial x^3}\Big(-\Div(\nabla (\widetilde g_0+G) \times (\nabla (\widetilde f_0+F)\times \nabla(\widetilde g_0+G)))+
\partial_f H(\widetilde f_0+F,\widetilde g_0+G)\Big)\cdot \frac{\partial^5 F}{\partial x^5} \,dx\,dy\,dz$$
\begin{equation}
\label{eq: for Delta4 F G}
+\int_\sP
-\frac{\partial^3}{\partial x^3}\Big(-\Div((\nabla (\widetilde f_0+F)\times \nabla (\widetilde g_0+G))\times \nabla (\widetilde f_0+F))
+\partial_g H(\widetilde f_0+F,\widetilde g_0+G)\Big)\cdot \frac{\partial^5G}{\partial x^5}\,dx\,dy\,dz.
\end{equation}
Thanks to \eqref{eq: estimates in r square} and $(F,G)\in \partial B(0,r)$, the terms containing partial derivatives of $H$ can be bounded below by $-\operatorname{Const} \xi r^2$. 
Also the terms in which $\widetilde f_0$ or $\tilde g_0$ is differentiated at least twice can be bounded below by $-\operatorname{Const} \xi r^2$ (by \eqref{eq: two estimates}).
Terms in which the factors $F$ or $G$ (or both) appear at least three times can too be bounded below by $-\operatorname{Const} r^3$.
Hence remain terms in which $\widetilde f_0$ and $\widetilde g_0$ are differentiated exactly once, the factors $F$ or $G$ (or both) appearing exactly twice and five times differentiated.
Furthermore terms that contain both factors $F$ or $G$ partially differentiated five times (in $x$, $y$ and/or $z$), but one of the two factors being at least differentiated once  in $y$ or $z$, can be bounded below by
$$-\operatorname{Const}\left(\sum_{m_1+m_2+m_3= 5,\,m_1\leq 4}\left\|\frac{\partial^{m_1+m_2+m_3}}{\partial x^{m_1}\partial y^{m_2}\partial z^{m_3}}(F,G)\right\|_{L^2(\sP)}\right)\left\|\frac{\partial^5}{\partial x^5}(F,G)\right\|_{L^2(\sP)}\,,$$
where the constant can change from line to line.
Hence
$$\eqref{eq: for Delta4 F G}\geq
\int_\sP\frac{\partial^4}{\partial x^4}(\nabla (\widetilde g_0+G) \times (\nabla (\widetilde f_0+F)\times \nabla(\widetilde g_0+G)))\cdot \nabla \frac{\partial^4 F}{\partial x^4} \,dx\,dy\,dz
$$$$
+\int_\sP\frac{\partial^4}{\partial x^4}((\nabla (\widetilde f_0+F)\times \nabla (\widetilde g_0+G))\times \nabla (\widetilde f_0+F))\cdot\nabla \frac{\partial^4G}{\partial x^4}\,dx\,dy\,dz
$$$$
-\operatorname{Const}\underbrace{\left(
\left(\sum_{m_1+m_2+m_3= 5,\,m_1\leq 4}\left\|\frac{\partial^{m_1+m_2+m_3}}{\partial x^{m_1}\partial y^{m_2}\partial z^{m_3}}(F,G)\right\|_{L^2(\sP)}\right)
\left\|\frac{\partial^5}{\partial x^5}(F,G)\right\|_{L^2(\sP)}
+ r^3+\xi r^2\right)}_{:=\operatorname{REST}} 
$$$$
$$$$
\geq B_{(\widetilde f_0+F,\widetilde g_0+G)}((\partial_x^4F,\partial_x^4G),(\partial_x^4F,\partial_x^4G))-\operatorname{Const}\cdot\operatorname{REST}
$$
$$\geq\int_{\sP} \Big( (\partial_{w_1}a_1)(\nabla( \widetilde f_0+F),\nabla( \widetilde g_0+G))(\partial_x^5F)^2+(\partial_{w_4}a_1)\partial_x^5F\partial_x^5G+(\partial_{w_1}a_4)\partial_x^5F\partial_x^5G+(\partial_{w_4}a_4)(\partial_x^5G)^2\Big)\, dx\, dy\, dz
$$$$
-\operatorname{Const}\cdot\operatorname{REST}$$
$$\stackrel{\operatorname{Prop.~}\ref{prop: simple inequality}}\geq
\int_{\sP}\frac{\widetilde v_{01}^2}
{|\nabla_{(y,z)} \widetilde f_0|^2+|\nabla_{(y,z)} \widetilde g_0|^2}
\Big((\partial_x^5 F)^2+(\partial_x^5 G)^2\Big)\,dx\,dy\,dz
-\operatorname{Const}\cdot \operatorname{REST}.
$$
Let $\rho_0\geq 0$  be defined by $$\rho_0=\left\|(F,G)\right\|_{L^2(\sP)}.$$
The following inequality holds if $m_1+m_2+m_3=5$ with $0\leq m_1\leq 4$, where $\epsilon>0$:
$$\left\|\frac{\partial^{m_1+m_2+m_3}}{\partial x^{m_1}\partial y^{m_2}\partial z^{m_3}}(F,G)\right\|_{L^2(\sP)}^2
$$
$$
\leq 
\epsilon  \left\|\frac{\partial^{5}}{\partial x^{5}}(F,G)\right\|_{L^2(\sP)}^2
+C_\epsilon  \left\|\frac{\partial^{5}}{\partial y^{5}}(F,G)\right\|_{L^2(\sP)}^2
+C_\epsilon  \left\|\frac{\partial^{5}}{\partial z^{5}}(F,G)\right\|_{L^2(\sP)}^2+\epsilon \rho_0^2;$$
see \eqref{eq: a Sobolev interpolation} (which remains true if the periodic variable $x_2$ is replaced by two periodic variables $(x_2,x_3)=(y,z)$).
Hence we get
$$
\int_\sP
-\frac{\partial^3}{\partial x^3}\Big(-\Div(\nabla (\widetilde g_0+G) \times (\nabla (\widetilde f_0+F)\times \nabla(\widetilde g_0+G)))+
\partial_f H(\widetilde f_0+F,\widetilde g_0+G)\Big)\cdot \frac{\partial^5 F}{\partial x^5} \,dx\,dy\,dz$$
$$
+\int_\sP
-\frac{\partial^3}{\partial x^3}\Big(-\Div((\nabla (\widetilde f_0+F)\times \nabla (\widetilde g_0+G))\times \nabla (\widetilde f_0+F))
+\partial_g H(\widetilde f_0+F,\widetilde g_0+G)\Big)\cdot \frac{\partial^5G}{\partial x^5}\,dx\,dy\,dz.
$$
\begin{equation}\label{eq: consequence of semi-coer bis}\geq
K_1||\partial_x^5(F,G)||_{L^2(\sP)}^2-\operatorname{Const} (r^3+\xi r^2)
-K_2\Big(||\partial_y^5(F,G)||_{L^2(\sP)}^2+||\partial_z^5(F,G)||_{L^2(\sP)}^2\Big)-\rho_0^{2}
\end{equation}
(taking $\epsilon>0$ small enough), where $K_1,K_2>0$ are additional constants.

\subsubsection{Fourth estimate from below}
Finally  set $\widetilde v_0=\nabla \widetilde f_0\times\nabla \widetilde g_0=(\widetilde v_{01},\widetilde v_{02},\widetilde v_{03})$. We have
$$
\int_\sP \Big(-\Div(\nabla (\widetilde g_0+G) \times (\nabla (\widetilde f_0+F)\times \nabla(\widetilde g_0+G)))+
\partial_f H(\widetilde f_0+F,\widetilde g_0+G)\Big) F\,dx\,dy\,dz
$$$$
+\int_\sP
\Big(-\Div((\nabla (\widetilde f_0+F)\times \nabla (\widetilde g_0+G))\times \nabla (\widetilde f_0+F))
+\partial_g H(\widetilde f_0+F,\widetilde g_0+G)\Big) G\,dx\,dy\,dz.
$$

After distributing products, the estimates are similar to the previous ones. We can immediately estimate the terms containing first-order partial derivatives of $H$ from below by $-\operatorname{Const} \xi r^2$ (see \eqref{eq: estimates in r square}).
Remaining terms in which the factors $F$ or $G$ (or both) appear at least three times
can also be bounded below by $-\operatorname{Const} r^3$.
Remaining terms in which $F$ or $G$ appears exactly once are bounded below by $-\operatorname{Const} \xi r^2$, thanks to \eqref{eq: two estimates}.
The terms that finally remain are such that the factors $F$ or $G$ (or both) appear exactly two times. 
Therefore, we get
$$\ldots\geq
B_{(\widetilde f_0,\widetilde g_0)}((F,G),(F,G))-\operatorname{Const} (r^3+\xi r^2)$$
$$\stackrel{\operatorname{~Thm }\ref{thm: quadratic part}}\geq
\frac{1}{2|\nabla \overline f|^2+2|\nabla \overline g|^2}\int_{\sP}\Big\{ 
\frac 1 {16}(\widetilde v_0\innpr \nabla F)^2+\frac 1{16}(\widetilde v_0\innpr \nabla G)^2+(1-O(\|\widetilde v_0'\|_{C(\overline \sP)}))\frac {\pi^2\min_{\overline \sP}  \widetilde v_{01}^2} {16L^2}(F^2+G^2) 
\Big\}\,dx\,dy\,dz
$$$$-\operatorname{Const} (r^3+\xi r^2)\,.$$
Hence
$$
\ldots\geq\frac 1 {2|\nabla \overline f |^2+2|\nabla \overline g|^2}
\int_{\sP} \frac {\pi^2\min_{\overline \sP}  \widetilde v_{01}^2} {32L^2}(F^2+G^2) 
\,dx\,dy\,dz
r-\operatorname{Const}(r^3+\xi r^2)
$$
\begin{equation}\label{eq: simplified semi-coer bis}
\geq K_1\rho_0^2-\operatorname{Const} (r^3+\xi r^2)
\end{equation}
for some constant $\operatorname{Const}>0$.
In all these estimates, given $\xi\in(0,1]$, we consider small enough $r>0$
(the positive constants $K_1,\operatorname{Const}$ being independent of $\xi,r>0$).

\subsection{Proof of Theorem \ref{thm: hydrodynamics}}
Let
$$s= 5,~~Y=Y_s=(H^s_{0,per}(\sP))^2,~~Y^*=Y_s^*=Y' \text{ (topological dual)},~~
$$
$$
~~V=V_s=(H^{s+2}_{0,per}(\sP))^2,~~V^*=V_s^*=V'=((H^{s+2}_{0,per}(\sP))')^2\,,$$
with the two canonical dualities $<\cdot,\cdot>_{Y\times Y^*}$ and $<\cdot,\cdot>_{V\times V^*}$.
The compact Sobolev injection $(H_{0,per}^s(\sP))^2\subset (C^{s-2}_{per}(\overline{\sP}))^2$ still holds.

We shall rely on Kato's theorem \ref{thm: Kato} by setting, for all $(F,G)\in Y$,
$$
A(F,G):=(\partial_{x_1}^s)^*\partial_{x_1}^{s-2}\widetilde{\operatorname{div}}(a(\nabla (\widetilde f_0+F),\nabla (\widetilde g_0 + G)))
-\lambda_1(-1)^s \partial_{x_2}^{s}\widetilde{\operatorname{div}}\partial_{x_2}\partial_{x_2}^{s-1}(a(\nabla (\widetilde f_0+F),\nabla (\widetilde g_0 + G)))
$$
\begin{equation}
\label{eq: choice of A bis}
-\lambda_1(-1)^s \partial_{x_3}^{s}\widetilde{\operatorname{div}}\partial_{x_3}\partial_{x_3}^{s-1}(a(\nabla (\widetilde f_0+F),\nabla (\widetilde g_0 + G)))
-\lambda_2\widetilde{\operatorname{div}}(a(\nabla (\widetilde f_0+F),\nabla (\widetilde g_0 + G)))
\end{equation}
$$
-(\partial_{x_1}^s)^*\partial_{x_1}^{s-2}(\nabla H)(\widetilde f_0+F,\widetilde g_0 + G)
+\lambda_1(-1)^s \partial_{x_2}^{s}\partial_{x_2}^{s}(\nabla H)(\widetilde f_0+F,\widetilde g_0 + G)
$$
\begin{equation}
\label{eq: choice of A ter}
+\lambda_1(-1)^s \partial_{x_3}^{s}\partial_{x_3}^{s}(\nabla H)(\widetilde f_0+F,\widetilde g_0 + G)
+\lambda_2(\nabla H)(\widetilde f_0+F,\widetilde g_0 + G)
\in V^* 
\end{equation}
for two constants $\lambda_1,\lambda_2>0$ to be chosen later.
Each term in \eqref{eq: choice of A bis} has to be interpreted as in \eqref{eq: choice of A}. In  \eqref{eq: choice of A ter},
$(-1)^s\partial_{x_i}^s:L^2(\sP)^2=(L^2(\sP)^2)'\rightarrow V'=V^*$ is the dual operator of  $\partial_{x_i}^s: V\rightarrow  L^2(\sP)^2$, $i\in\{2,3\}$.

\begin{prop}
\label{prop: verification semi-coercivity bis}
There exist $\lambda_1,\lambda_2,L>0$ such that for all $W$ in $V$ with $||W||_Y$ small enough:
$$\langle W,AW\rangle_{V\times V^*}\geq L ||W||_Y^2
\, -\, L^{-1}||W||_Y^3\,.$$
\end{prop}
\begin{proof}
As in Corollary \ref{cor: controle a},
\begin{equation}
\label{eq: controle a bis}
\partial_{x_2}^{s}(a(\nabla (\widetilde f_0+F),\nabla (\widetilde g_0 + G)))\in L^2(\sP)^6
~~\text{ and }~~
\partial_{x_3}^{s}(a(\nabla (\widetilde f_0+F),\nabla (\widetilde g_0 + G)))\in L^2(\sP)^6
\end{equation}
for all $(F,G)\in V$.
Write 
$$W=(F,G)\in \mathbb R^2\,,~~  \nabla W=\Big(\nabla F,\nabla  G\Big)\in \mathbb R^6
~~\text{ and }~~\widetilde{ \nabla}W=\Big(\nabla (\widetilde f_0+F),\nabla (\widetilde g_0 + G)\Big)\in \mathbb R^6,$$
and let $\lambda_1,\lambda_2>0$ be chosen later.
For $s=5$, $W\in V$, $\rho_0=\left\|(F,G)\right\|_{L^2(\sP)}$ and $r:=||W||_{Y}$ small enough:
$$\langle W,AW\rangle_{V\times V^*}
=\left\langle W,(\partial_{x_1}^s)^*\partial_{x_1}^{s-2}\widetilde{\operatorname{div}}(a(\widetilde{\nabla }W))
-\lambda_1(-1)^s \partial_{x_2}^{s}\widetilde{\operatorname{div}}\partial_{x_2}\partial_{x_2}^{s-1}(a(\widetilde{\nabla }W))\right.
$$$$
-\left.\lambda_1(-1)^s \partial_{x_3}^{s}\widetilde{\operatorname{div}}\partial_{x_3}\partial_{x_3}^{s-1}(a(\widetilde{\nabla }W))
-\lambda_2\widetilde{\operatorname{div}}(a(\widetilde{\nabla} W))\right\rangle _{V\times V^*}
$$$$
+\left\langle W,-(-1)^s\partial_{x_1}^s\partial_{x_1}^{s-2}(\nabla H)(\widetilde f_0+F,\widetilde g_0 + G)
+\lambda_1(-1)^s \partial_{x_2}^{s}\partial_{x_2}^{s}(\nabla H)(\widetilde f_0+F,\widetilde g_0 + G)\right.
$$$$
+\left.\lambda_1(-1)^s \partial_{x_3}^{s}\partial_{x_3}^{s}(\nabla H)(\widetilde f_0+F,\widetilde g_0 + G)
+\lambda_2(\nabla H)(\widetilde f_0+F,\widetilde g_0 + G)\right\rangle _{V\times V^*}
$$$$
\stackrel {\eqref{eq: controle a bis}} =\int_{\sP}\left(\partial_{x_1}^s W\right)\cdot\left(\partial_{x_1}^{s-2}\widetilde{\operatorname{div}}(a(\widetilde{\nabla} W))\right)\, dx_1dx_2
+\lambda_1\int_{\sP}\left(\partial_{x_2}^{s}\nabla W\right)\cdot \left(\partial_{x_2}^{s}(a(\widetilde{\nabla} W))\right)\, dx_1dx_2
$$$$
+\lambda_1\int_{\sP}\left(\partial_{x_3}^{s}\nabla W\right)\cdot \left(\partial_{x_3}^{s}(a(\widetilde{\nabla} W))\right)\, dx_1dx_2
+\lambda_2\int_{\sP}(\nabla W)\cdot a(\widetilde{\nabla} W)\, dx_1dx_2
	-\operatorname{Const}(1+2\lambda_1+\lambda_2)\xi r^2
$$$$
\stackrel{\eqref{eq: consequence of semi-coer bis},\eqref{eq: conclusion bis},\eqref{eq: conclusion ter},\eqref{eq: simplified semi-coer bis}}
\geq
K_1||\partial_x^5(F,G)||_{L^2(\sP)}^2
-K_2\Big(||\partial_y^5(F,G)||_{L^2(\sP)}^2+||\partial_z^5(F,G)||_{L^2(\sP)}^2\Big)-\rho_0^{2}
$$$$
+ \lambda_1 K_1||\partial_y^5(F,G)||_{L^2(\sP)}^2
+ \lambda_1 K_1||\partial_z^5(F,G)||_{L^2(\sP)}^2+\lambda_2 K_1\rho_0^2
-\operatorname{Const}(1+2\lambda_1+\lambda_2) (r^3+\xi r^2)\,.
$$
Again the norms defined by the square roots of
$\sum_{i+j+k\leq s}||\partial_{x_1}^i\partial_{x_2}^j\partial_{x_3}^kW||_{L^2(\sP)}^2$ 
and 
$||W||_{L^2(\sP)}^2+||\partial_{x_1}^sW||_{L^2(\sP)}^2+||\partial_{x_2}^sW||_{L^2(\sP)}^2 +||\partial_{x_3}^sW||_{L^2(\sP)}^2$ 
are equivalent on $Y$.
Thus
$$\langle W,AW\rangle_{V\times V^*}\geq L||W||_Y^2
\, -\, L^{-1} ||W||_Y^3$$
	for  $\lambda_1,\lambda_2>0$ chosen large enough (and by considering small enough $\xi>0$ and then small enough $r>0$).
\end{proof}

By Proposition \ref{prop: verification semi-coercivity bis},  there exist $\lambda_1,\lambda_2, r>0$ such that
$$\beta:=\inf\left\{<W,AW>_{V\times V^*}:\, W\in V\,,  ||W||_Y=r \right\}\geq \frac L 2 r^2>0$$
if $L^{-1}r^3\leq \frac 1 2 L r^2$.
By an analogous result to Lemma \ref{lemma: weak sequential continuity},
$A:Y \rightarrow  V^*$  is weakly sequentially continuous.
The hypotheses of Theorem \ref{thm: Kato} being satisfied with  the closed ball
$$K=\{W\in Y:\, ||W||_Y\leq r\},$$ 
$A(K)$ contains the closed ball centered at the origin $0\in Y^*$ of radius $\beta r^{-1}>0$ in $Y^*$.
In particular there exists $U=(F,G)\in K$ such that $AU=0$.
By an analogous result to Proposition \ref{prop: Aw=0 w=0}, 
$$\widetilde{\operatorname{div}}\Big(a(\widetilde{\nabla} U)\Big)-(\nabla H)(\widetilde f_0+F,\widetilde g_0+G)=0.$$

\section{Uniqueness and minimization principle}
\label{section: minimizer}
In this section, we prove that the solutions obtained in Theorems \ref{thm: hydrodynamics} and \ref{thm: main} are locally unique, as local minimizers of strictly convex functionals.\medskip

Let us start with the toy model. Thanks to hypothesis \eqref{eq: symmetry}, there exists $\Gamma\in C^{s+1}(\mathbb R^4,\mathbb R)$, 
$$(x_1,x_2,z_1,z_2)=(x,z)\rightarrow \Gamma(x,z),$$ 
such that  
$$a(x,y)=(a_1(x,z),a_2(x,z))=\Big(\partial_{z_1}\Gamma(x,z),\partial_{z_2}\Gamma(x,z)\Big).$$ 
It is therefore possible to consider the integral functional
$$u\rightarrow \int_{\sP}\Big(\Gamma(x,\nabla u(x))+h(x)u(x)\Big)\, dx_1dx_2,
~~u\in H^s_{0,per}(\sP)\subset C^3(\overline{\sP})~~(s\geq 5).$$
As convergence in $ H^s_{0,per}(\sP)$ implies the uniform convergence of the two first-order partial derivatives, the integral functional is of class $C^2$, its Fr\'echet derivative at $u\in H^s_{0,per}(\sP)$ being the linear map
$$v\rightarrow \int_{\sP}\Big(a(x,\nabla u(x))\cdot \nabla v+h(x)v(x)\Big)\, dx_1dx_2
=\int_{\sP}\Big(-\operatorname{div}\Big(a(x,\nabla u(x))\Big) v+h(x)v(x)\Big)\, dx_1dx_2,
$$
for $v\in H^s_{0,per}(\sP)$.
Hence every critical point of the integral functional satisfies \eqref{eq: the toy equation}.
Its second order Fr\'echet derivative at $u\in H^s_{0,per}(\sP)$ is the symmetric bilinear map
$$(v,w)\rightarrow \int_{\sP}\sum_{i=1}^2\sum_{p=1}^2(\partial_{z_p}a_i)(x,\nabla u(x)) (\partial_{x_i}v)(x) (\partial_{x_p}w)(x)\, dx_1dx_2
$$
for $(v,w)\in H^s_{0,per}(\sP)^2$.
By \eqref{eq: hyp on a}, the integral functional is strictly convex in a small open ball in $H^s_{0,per}(\sP)$ centered at the origin, because
$$\int_{\sP}\sum_{i=1}^2\sum_{p=1}^2(\partial_{z_p}a_i)(x,\nabla u(x)) (\partial_{x_i}v)(x) (\partial_{x_p}v)(x)\, dx_1dx_2
\geq\rho \int_{\sP}(\partial_{x_1}v)^2 \, dx_1dx_2>0$$
if $v$ does not vanish, thanks to the boundary condition $v(0,x_2)=v(1,x_2)=0$. 
If $r>0$ is small enough, the solution obtained in Theorem \ref{thm: main} is therefore locally unique and a local minimizer.\medskip

We now turn to the hydrodynamic problem. As observed in \cite{BuWa}, the integral functional already introduced in \eqref{eq: integral functional},
$$(F,G)\rightarrow \int_{\sP}\left\{ \frac12\left|\nabla(\widetilde f_0+F)\times \nabla( \widetilde g_0+G)\right|^2+H(\widetilde f_0+F,\widetilde g_0+G)\right\}\,dx\,dy\,dz
$$
for $(F,G)\in H^5_{0,per}(\sP)^2\subset C^3(\overline{\sP})^2$,
is strictly convex in a small open ball $H^5_{0,per}(\sP)^2$ centered at the origin, thanks to \eqref{eq: estimates in r square} (with $\xi>0$ first chosen small enough) and Theorem \ref{thm: quadratic part}. Hence, if $r>0$ is small enough, the solution obtained in Theorem \ref{thm: hydrodynamics} is also locally unique and a local minimizer. Note that it is also possible to treat this hydrodynamic problem by a global minimization approach \cite{Bu:2012}, leading to generalized flows.

\end{document}